\documentclass[11pt]{amsart}
\usepackage{graphicx,subfigure,epsfig}
\usepackage{subfigure}
\usepackage{epstopdf}
\usepackage{mathtools}

\usepackage{amssymb,amsmath}
\usepackage{caption}
\usepackage{float}
\usepackage[usenames, dvipsnames]{color}
\usepackage{bm}
\usepackage{enumitem}
\newtheorem{theorem}{Theorem}[section]

\newtheorem{lemma}{Lemma}[section]
\newtheorem{remark}{Remark}

\newtheorem{corollary}{Corollary}[section] 
\newtheorem{example}{Example}[section] 
\newtheorem{proposition}{Proposition}[section]
\newtheorem{problem}{Problem}[section]
\newtheorem{question}{Question}[section]

\usepackage{placeins}

\def\boo{b_{11}}
\def\bot{b_{12}}
\def\btt{b_{22}}
\def\bto{b_{21}}
\def\bc{b_{23}}
\newcommand{\Z}{\mathbb{Z}}

\title[Representations of virtual and universal braid groups]{Virtual and universal braid groups,  their quotients and representations}
\author[]{V. Bardakov, I.~Emel'yanenkov, M.~Ivanov, T.~Kozlovskaya, T.~Nasybullov, A.~Vesnin}
\date{\today}
\usepackage{graphicx}

\begin{document}
\maketitle 

\begin{abstract}
In the present paper we study structural aspects of certain quotients of braid groups and virtual braid groups. In particular, we construct and study linear  representations $B_n\to {\rm GL}_{n(n-1)/2}\left(\mathbb{Z}[t^{\pm1}]\right)$, $VB_n\to {\rm GL}_{n(n-1)/2}\left(\mathbb{Z}[t^{\pm1}, t_1^{\pm1},t_2^{\pm1},\ldots, t_{n-1}^{\pm1}]\right)$
which are connected with the famous Lawrence-Bigelow-Krammer representation. It turns out that these representations are faithful representations of crystallographic groups $B_n/P_n'$, $VB_n/VP_n'$, respectively. Using these representations we study certain properties of the groups $B_n/P_n'$, $VB_n/VP_n'$. Moreover, we construct new representations and decompositions of universal braid groups $UB_n$.

 ~
 
 ~
\end{abstract}

\tableofcontents

\section*{Introduction}
Braid groups were introduced by E. Artin as a tool for working with classical knots and links. The
famous Alexander theorem says that every classical link is equivalent to the closure of some braid,
and the well-known Markov theorem describes braids which have equivalent closures.

In 1999 L. Kauffman introduced the virtual knot theory which generalizes the classical knot theory \cite{Ka}. Virtual braid groups were introduced in the same paper as a tool for working with virtual knots and links.
The analogues of the Alexander and the Markov theorems for virtual braids and links were formulated and proved by S. Kamada in \cite{Kam1}. Another form of the Markov theorem for virtual braids
and links was introduced by Kauffman and Lambropoulou in \cite{Kalo}.

Welded braid groups, unrestricted virtual braid groups, flat virtual braid groups, flat welded braid groups, Gauss virtual braid groups are quotients of virtual braid groups. These groups have various applications in algebra and topology, in particular, each of these groups corresponds to some knot theory, which has a geometric interpretation that is often very useful when studying applications. Welded and unrestricted virtual braid groups are used for working with welded links \cite{FRR} and fused links \cite{ABMW,Nas}, respectively. Flat virtual braid groups were introduced in \cite{Ka2} for working with flat virtual knots and links, which were introduced in \cite{Ka}, and which model curves on surfaces. The notion of flat virtual links is the same to the notion of virtual strings introduced by V. Turaev in \cite{Tur}. In the papers \cite{IvaVes,KPV} invariants for flat virtual knots and links were used for constructing invariants for virtual knots and links. Gauss virtual braid groups are used for studying Gauss virtual links, which are called ``homotopy classes of
Gauss words'' by Turaev in \cite{Tur2}, and called ``free knots'' by Manturov in \cite{Mant}. All together classical braid groups, virtual braid groups, welded braid groups, unrestricted virtual braid groups, flat virtual braid groups, flat welded braid groups and Gauss virtual braid groups are called braid-like groups. Universal braid groups were introduced in \cite{Bar}. The name ``universal'' in this paper is motivated by the fact that universal braid groups contain classical braid groups, and moreover these groups have singular braid groups and all braid-like groups as quotients. Hence the structure of these groups play an important role in the study of all braid-like groups.

One way of studying braid-like groups is to study representations of these groups by automorphisms of different algebraic systems. Using the representations one can prove that various algorithmic problems (the word problem, the conjugacy problem etc) can be solved in these groups. Also using representations of  braid-like groups one can construct invariants for the corresponding knot theories \cite{BarNas1,BarNas2,BarNas3}. 

One of the first representations of braid-like groups is the Artin's representation $\varphi:B_n\to{\rm Aut}(F_n)$ of the braid group $B_n$ by automorphisms of the free group $F_n$ with $n$ generators. The Artin's representation is faithful. During the years a lot of representations of braid-like
groups were investigated. For example, there are the Silver-Williams representation $\varphi_{SW}:VB_n\to {\rm Aut}(F_{n}*\mathbb{Z}^{n+1})$ (see \cite{SilWil}), the Boden-Dies-Gaudreau-Gerlings-Harper-Nicas representation $\varphi_{BD}:VB_n\to {\rm Aut}(F_{n}*\mathbb{Z}^2)$ (see \cite{BDGGHN}),  the Kamada representation $\varphi_K:VB_n\to {\rm Aut}(F_{n}*\mathbb{Z}^{2n})$ (see \cite{BN}), the representations $\varphi_M:VB_n\to {\rm Aut}(F_n*\mathbb{Z}^{2n+1})$, $\tilde{\varphi}_M:VB_n\to {\rm Aut}(F_n*\mathbb{Z}^n)$ of Bardakov-Mikhalchishina-Neshchadim (see \cite{BarMikNes, BarMikNes2}). In the paper~\cite{alotofauthors} there are representations of the flat virtual braid groups by automorphisms of the free group, which do not preserve forbidden relations. 

The study of linear representations of braid groups was started when Burau  constructed his famous representation $\varphi_B:B_n\to{\rm GL}_n(\mathbb{Z}[t,t^{-1}])$ (see, for example, \cite[Section~3]{Bir}). For $n=2,3$ the Burau representation is known to be faithful, for $n\geq 5$ the Burau representation is known to have non-trivial kernel \cite{Big}, for $n=4$ the question on faithfulness of the Burau representation is still open (see some positive results, for example, in \cite{trac}). The Burau representation was generalized in 1961 by Gassner \cite{Gas} to a related representation of the pure braid group $\varphi_G:P_n\to {\rm GL}_n(\mathbb{Z}[t_1^{\pm1},t_2^{\pm1},\dots,t_n^{\pm1}])$. It is not known whether the Gassner representation is faithful for $n>3$. Some new results about representations of virtual braid groups can be found, for example, in \cite{BarNas1, BarNas3}. Certain linear representations of flat virtual braid groups and Gauss virtual braid groups are studied in \cite{alotofauthors}.

Another way to study braid-like groups is to explore normal forms of elements in braid-like groups which usually (but not always) arise when a braid-like group can be written as a product (for example, direct product, semiderict product, free product etc). There are several normal forms of elements in classical braid groups which can be found, for example, in  \cite{Kassel}. Attempts to construct a normal form for elements of the virtual braid groups were made in the paper \cite{Bar}. Paper \cite{BarBelDom} contains results about the normal form of elements in unrestricted virtual braid groups and flat welded braid groups. 

In the present paper we study structural aspects of certain quotients of braid groups and virtual braid groups using both representations and decompositions. In particular, we construct and study linear  representations 
\begin{align*}
B_n\to {\rm GL}_{n(n-1)/2}\left(\mathbb{Z}[t^{\pm1}]\right),&& VB_n\to {\rm GL}_{n(n-1)/2}\left(\mathbb{Z}[t^{\pm1}, t_1^{\pm1},t_2^{\pm1},\ldots, t_{n-1}^{\pm1}]\right)
\end{align*}
which are connected with the famous Lawrence-Bigelow-Krammer representation. It turns out that these representations are faithful representations of crystallographic groups $B_n/P_n'$, $VB_n/VP_n'$, respectively. Using these representations we study certain properties of the groups $B_n/P_n'$, $VB_n/VP_n'$. Moreover, we construct new representations and decompositioms of universal braid groups.

The paper is organized as follows. In Section~\ref{sec1} we recall definitions and basic properties of braid groups, virtual braid groups, welded braid groups and pure subgroups of these groups. In Section~\ref{sec2} we construct a faithful linear representation of the crystallographic group $B_n / P_n'$ that is the quotient of the braid group $B_n$ by the commutator subgroup $P_n' = [P_n, P_n]$ of the pure braid group $P_n$ (Theorem~\ref{theorem2.1}). The torsion in this group is studied in \cite{GGO}. Using the constructed representation we describe all elements of order three in $B_3/ P_3'$ (Proposition~\ref{proposition2.1}). Note that this group has only three-torsion. 
In Section~\ref{sec3} we construct a faithful linear representation of another crystallographic group $VB_n / VP_n'$ which contains $B_n / P_n'$ (Theorem~\ref{theorem3.1}). In Proposition~\ref{proposition3.2} we demonstrate that $VB_3 / VP_3'$ contains infinitely many elements of order two. Note that $B_n / P_n'$  has no two-torsion. In Section~\ref{sec4} we show that $VP_3$ can be decomposed into the free product of infinite cyclic group and group that is a semi-direct product of free group of rank three and free group of rank two (Proposition~\ref{proposition4.1}). In Theorem~\ref{theorem4.1} we prove that 
$VB_3$ is linear. In Section~\ref{sec5} we discuss the universal braid group $ UB_{n}$ introduced in~\cite{Bar}, and its pure subgroup $UP_{n}$. In particular we construct several representations of  $UP_{n}$ by automorphisms of free groups which act faithfully on a subgroup $T_n$ of $UP_n$. Also in Theorem~\ref{theorem5.1} we demonstrate that the groups $UP_{2}$ and $UP_{3}$ admit decompositions into free products.

\section{Preliminaries} \label{sec1}
In this section we give necessary preliminaries. We use classical notations. If $G$ is a group, and $a,b\in G$, then we denote by $a^b=b^{-1}ab$ the conjugate of $a$ by $b$, and by $[a,b]=a^{-1}b^{-1}ab$ the commutator of the elements $a,b$. If $X$ is a set, then we denote by ${\rm Sym}(X)$ the set of all bijections from $X$ to~$X$. All actions are supposed to be right, i.~e. if $f,g\in {\rm Sym}(X)$, then for $x\in X$ we denote $fg(x)=g(f(x))$.

\subsection{Braid groups and conjugating automorphism groups}\label{braidsandautomorphisms}
For $n\geq 2$ the braid group $B_n$ on $n$ strands is defined as
the group generated by $(n-1)$ elements
$\sigma_1,\sigma_2,\ldots,\sigma_{n-1}$
with defining relations
\begin{align}
 \label{eq1}\sigma_i\sigma_{i+1}\sigma_i& = \sigma_{i+1}\sigma_i\sigma_{i+1},&&i=1,2,\ldots,n-2,\\
\label{eq2}\sigma_i\sigma_j &= \sigma_j\sigma_i,&& |i-j|\geq 2. 
\end{align}
Relations (\ref{eq1}) and (\ref{eq2}) are called \emph{braid relations}.  There is a homomorphism $\varphi$ from $B_n$ to the symmetric group $S_n$ on $n$ symbols $\{ 1, 2, \ldots, n\}$, which maps the generator $\sigma_i$ to the transposition  $(i,i+1)$ for $i=1,2,\ldots,n-1$. The kernel $\operatorname{Ker}(\varphi)$ is called the \emph{pure braid group} on $n$ strands and is denoted by $P_n$. From the definition it is clear that we have the following short exact sequence:  
$$
1 \to P_{n} \to B_{n} \stackrel{\varphi}{\to} S_{n} \to  1.
$$ 
It is known that the pure braid group $P_n$ can be generated by the elements $a_{i,j}$ for $1\leq i < j\leq n$, where
\begin{align*}
a_{i,i+1}&=\sigma_i^2, &&i=1,2,\dots,n-1,\\
a_{i,j}&= \sigma_{j-1} \, \sigma_{j-2} \ldots \sigma_{i+1} \, \sigma_i^2 \, \sigma_{i+1}^{-1} \ldots \sigma_{j-2}^{-1} \, \sigma_{j-1}^{-1},&&i+1< j \leq n.
\end{align*}
Also it is know that the pure braid group $P_{n}$ can be decomposed into a semi-direct product $P_n = U_{n} \rtimes P_{n-1}$, where $U_n$ is the free group with free generators $a_{1,n}, a_{2,n},\ldots,a_{n-1,n}$, and $P_{n-1}$ is the pure braid group on $(n-1)$ strands. Here $U_{n}$ is a normal subgroup of $P_{n}$. Similar, $P_{n-1}$ is a semi-direct product of the free group $U_{n-1}$ with the free generators $a_{1,n-1}, a_{2,n-1} \ldots,a_{n-2,n-1}$ and $P_{n-2}$, and so on. Therefore $P_n$ has the following decomposition into the iterated semi-direct product (see~\cite{Mar}):
$$
P_n=U_n\rtimes (U_{n-1}\rtimes (\ldots \rtimes (U_3\rtimes U_2))\ldots), 
$$
where  $U_i\simeq F_{i-1}$ is the free group on $(i-1)$ generators for $i=2,3,\ldots,n$. 

The group $B_n$ has a faithful representation by automorphisms of the free group $F_n = \langle x_1, x_2, \ldots, x_n \rangle$, where the generator $\sigma_i$, $i=1,2,\ldots,n-1$, defines the automorphism (which we denote by the same symbol  $\sigma_{i} \in \operatorname{Aut} (F_{n})$) which acts on the generators $x_1,x_2,\dots,x_n$ of $F_n$ by the following rule
$$
\sigma_{i} : \begin{cases}
x_{i} \mapsto x_{i} x_{i+1}x_i^{-1},&\\
x_{i+1} \mapsto x_{i},&\\
x_{l} \mapsto x_{l}, &  l \neq i,i+1.
\end{cases}
$$
By the Artin's theorem \cite[Theorem 1.9]{Bir} an automorphism $\beta$ from $\operatorname{Aut}(F_n)$ belongs to $B_n < \operatorname{Aut} (F_{n})$ if and only if $\beta$ satisfies the following two conditions:
\begin{enumerate}
\item[1)]  $\beta(x_i) = a_i^{-1} \, x_{\pi(i)} \, a_i$  for  $1\leq i\leq n$, 
\item[2)]  $\beta(x_1x_2  \cdots x_n)=x_1x_2 \cdots x_n$,
\end{enumerate}
where $\pi $ is a permutation from $S_n$ and $a_i \in F_n$ for $i=1,2,\dots,n$.

The automorphism $\beta$ of the free group $F_n$ is called a  \emph{conjugating automorphism} (or permutation-conjugating by terminology from  \cite{FRR}) if it satisfies condition 1) from the list above. It is clear that every automorphism from $B_n<{\rm Aut}(F_n)$ is a conjugating automorphism. The group of all conjugating automorphisms of $F_n$ is denoted by $C_n$. It is known that this group is generated by the automorphisms $\sigma_i$ for $i=1,2,\ldots,n-1$ and automorphisms $\alpha_i$ for $i=1, 2, \ldots,$ $n-1$, where $\alpha_i$ acts on the generators $x_1,x_2,\dots,x_n$ of $F_n$ by the following rule
$$
\alpha_{i} : \begin{cases}
x_{i} \mapsto x_{i+1},&\\
x_{i+1} \mapsto x_{i},  &\\
x_{l} \mapsto x_{l}, &  l \neq i, i+1.
\end{cases} 
$$
Generators $\sigma_{1},\sigma_2,\ldots,\sigma_{n-1}$ of $C_{n}$ satisfy the braid relations (\ref{eq1}) and (\ref{eq2}) of $B_{n}$. Also it is easy to  see that the automorphisms $\alpha_1,\alpha_2,\dots,\alpha_{n-1}$ of $C_n$ generate the symmetric group $S_n$ in $C_n$ which has the following defining relations  
\begin{align}
\label{eq3}\alpha_i \alpha_{i+1} \alpha_i &= \alpha_{i+1} \alpha_i \alpha_{i+1},&& i=1,2,\ldots,n-2, \\
\label{eq4}\alpha_i \alpha_j &= \alpha_j \alpha_i,&& |i-j|\geq 2,\\
\label{eq5}\alpha_i^2 &= 1,&&  i=1,2,\ldots,n-1. 
\end{align}
Moreover, the generators $\sigma_{1},\sigma_2,\ldots,\sigma_{n-1},\alpha_{1},\alpha_2,\dots,\alpha_{n-1}$ of $C_{n}$ satisfy the following mixed relations (see \cite{FRR, Sav}): 
\begin{align}
\label{eq6}\alpha_i \sigma_j &= \sigma_j \alpha_i,&&  |i-j|\geq 2, \\
\label{eq7}\sigma_i  \alpha_{i+1} \alpha_i &= \alpha_{i+1}  \alpha_i  \sigma_{i+1},&&  i=1,2,\ldots,n-2, \\
 \label{eq8}\sigma_{i+1} \sigma_{i} \alpha_{i+1} &= \alpha_{i} \sigma_{i+1} \sigma_{i}, && i=1,2,\ldots,n-2.
\end{align}
Summarizing, the group $C_{n}$ has  generators $\sigma_{1},\sigma_2,\ldots,\sigma_{n-1},\alpha_{1},\alpha_2,\dots,\alpha_{n-1}$ with the defining relations (\ref{eq1})-(\ref{eq8}).

For $1\leq i\neq j \leq n$ denote by  $\varepsilon_{i,j}$ the automorphism of $F_n$ which acts on the generators of $F_n$ by the rule
$$
\varepsilon_{i,j}: \begin{cases}
x_{i} \mapsto x_{j}^{-1} x_i x_j, &  i\neq j, \\
x_{l} \mapsto x_{l}, &  l\neq i.
\end{cases}
$$
The subgroup of ${\rm Aut}(F_n)$ generated by  $\varepsilon_{i,j}$ for $1\leq i\neq j \leq n$ is called the \emph{group of basis-conjugating automorphisms} and is denoted by $Cb_n$. It is clear that every generator $\varepsilon_{i,j}$ belongs to $C_n$, therefore the group $Cb_n$ of basis-conjugating automorphisms is a subgroup of the group $C_n$ of conjugating automorphisms. In~\cite{Mac} it is proved that $Cb_n$ has the following defining relations: 
\begin{align}
\label{eq9}\varepsilon_{i,j} \varepsilon_{k,l} &= \varepsilon_{k,l}  \varepsilon_{i,j},\\
\label{eq10}\varepsilon_{i,j} \varepsilon_{k,j} &= \varepsilon_{k,j} \varepsilon_{i,j},\\
\label{eq11}(\varepsilon_{i,j} \varepsilon_{k,j}) \varepsilon_{i,k} &= \varepsilon_{i,k}  (\varepsilon_{i,j} \varepsilon_{k,j}),
\end{align} 
where different letters denote different indices. 

It is easy to see that the group $C_n$ of conjugating automorphisms can be decomposed into a semi-direct product $C_n = Cb_n \rtimes S_n$, where the permutation group $S_n$ is generated by the automorphisms $\alpha_1, \alpha_2, \ldots, \alpha_{n-1}$. In \cite{Sav} it is prove that the following relations hold: 
\begin{align*}
\varepsilon_{i,i+1} &= \alpha_i \sigma^{-1}_i,&&\\  
\varepsilon_{i+1,i} &= \sigma^{-1}_i \alpha_i,&&\\
\varepsilon_{i,j} &= \alpha_{j-1} \alpha_{j-2} \ldots \alpha_{i+1}  \varepsilon_{i, i+1} \alpha_{i+1} \ldots \alpha_{j-2} \alpha_{j-1},&&\text{if}~i <j,\\
\varepsilon_{j,i} &= \alpha_{j-1} \alpha_{j-2} \ldots \alpha_{i+1}  \alpha_i \varepsilon_{i, i+1} \alpha_i \alpha_{i+1} \ldots \alpha_{j-2}  \alpha_{j-1},&&\text{if}~i <j.
\end{align*}
The structure of $Cb_n$ was studied in \cite{Bar, BarP}, where it was shown that  $Cb_n$, $n\geq 2$, has a decomposition into an iterated semi-direct product 
$$
Cb_n = D_{n-1}\rtimes (D_{n-2}\rtimes (\ldots.  \rtimes (D_2\rtimes D_1))\ldots ),
$$
where $D_i$ denotes a subgroup of $Cb_n$ generated by the elements
$$
\{ \varepsilon_{i+1,1}, \varepsilon_{i+1,2}, \ldots,\varepsilon_{i+1,i}, \varepsilon_{1,i+1}, \varepsilon_{2,i+1}, \ldots \varepsilon_{i,i+1} \},
$$ 
where $i=1,2,\ldots,n-1,$
The elements $\{ \varepsilon_{i+1,1}, \varepsilon_{i+1,2}, \ldots,\varepsilon_{i+1,i} \}$  generate the free group of rank $i$ in $D_i$, while the elements $\{ \varepsilon_{1,i+1}, \varepsilon_{2,i+1}, \ldots, \varepsilon_{i,i+1} \}$ generate the free abelian group of rank $i$ in $D_i$.

The  pure braid group $P_n$ is a subgroup of $Cb_n$. The generators $a_{i,j}$ (for $1\leq i < j\leq n$) of $P_n$ can be written in $Cb_n$ in the following form
\begin{align*}
a_{i,i+1}&=\varepsilon_{i,i+1}^{-1}  \varepsilon_{i+1,i}^{-1}
\end{align*}
for  $i=1,2,\ldots,n-1$, and
\begin{align*}
a_{i,j} & = \varepsilon_{j-1,i} \varepsilon_{j-2,i} \ldots \varepsilon_{i+1,i} (\varepsilon_{i,j}^{-1} \varepsilon_{j,i}^{-1})  \varepsilon_{i+1,i}^{-1} \ldots \varepsilon_{j-2,i}^{-1}  \varepsilon_{j-1,i}^{-1} \\ 
& = \varepsilon_{j-1,j}^{-1} \varepsilon_{j-2,j}^{-1} \ldots \varepsilon_{i+1,j}^{-1}  (\varepsilon_{i,j}^{-1}  \varepsilon_{j,i}^{-1})  \varepsilon_{i+1,j} \ldots \varepsilon_{j-2,j}  \varepsilon_{j-1,j},
\end{align*}
for $2 \leq i+1 < j \leq n$.

\subsection{The singular braid monoid}
\textit{The Baez-Birman monoid}
\cite{Baez, Bir93} or \textit{the singular braid monoid} $SB_n$ is defined as the monoind with the generators $\sigma_1,\sigma_2,\dots,\sigma_{n-1}$, $\sigma_1^{-1},\sigma_2^{-1},\dots,\sigma_{n-1}^{-1}$, $\tau_1,\tau_2,\dots,\tau_{n-1}$, and definig relations (\ref{eq1}), (\ref{eq2}) and addinitonal relations
\begin{align}	
\sigma_i \, \sigma_i^{-1} &= \sigma_i^{-1}\sigma_i=1,&&i=1,2,\dots,n-1, \label{eq1211}\\
\tau_i \, \tau_j &= \tau_j \, \tau_i,&&|i - j| \geq 1, \label{eq121}\\
\tau_{i} \, \sigma_{j} &= \sigma_{j} \, \tau_{i},&&|i - j| \geq 1, \label{eq131}\\
\tau_{i}  \, \sigma_{i} &= \sigma_{i} \, \tau_{i},&&i=1,2,\ldots,n-1,  \label{eq141}\\
\sigma_{i} \, \sigma_{i+1} \, \tau_i &= \tau_{i+1} \, \sigma_{i}  \, \sigma_{i+1},&&i=1,2,\ldots,n-2, \label{eq151}\\
\sigma_{i+1}  \, \sigma_{i} \, \tau_{i+1} &= \tau_{i} \,
\sigma_{i+1} \, \sigma_{i},&& i=1,2,\ldots,n-2.
 \label{eq161}
\end{align}

In the work \cite{FKR} it was proved that the singular braid monoid $SB_n$ is embedded into the group $SG_n$, which is called the 
singular braid group and has the same defining relations as $SB_n$.

\subsection{The virtual braid group and the welded braid group}\label{virtbraid}

The \emph{virtual braid group} $VB_n$ on $n$ strands, introduced in \cite{Ka}, is the group with the generators $\sigma_1,\sigma_2,\dots,\sigma_{n-1}$, $\rho_1,\rho_2,\dots,\rho_{n-1}$, which is defined by the relations (\ref{eq1})-(\ref{eq2}) and additional relations
 \begin{align}
\label{eq12} \rho_{i}  \rho_{i+1} \rho_{i} &= \rho_{i+1} \rho_{i} \rho_{i+1},&&  i=1,2,\ldots,n-2,\\
\label{eq13}\rho_{i}  \rho_{j} &= \rho_{j}  \rho_{i},&&  |i-j| \geq 2,   \\
\label{eq14}\rho_{i}^2 &= 1,&&  i=1,2,\ldots,n-1,\\ 
\label{eq15}\sigma_{i} \rho_{j} &= \rho_{j}  \sigma_{i},&& |i-j| \geq 1,\\
\label{eq16}\rho_{i} \rho_{i+1}  \sigma_{i} &= \sigma_{i+1}  \rho_{i}   \rho_{i+1},&&  i=1,2,\ldots,n-2. 
\end{align}
The elements $\sigma_1,\sigma_2,\dots,\sigma_{n-1}$ generate the braid group $B_n$ with the braid relations (\ref{eq1})-(\ref{eq2}) in $VB_n$. Relations (\ref{eq12})-(\ref{eq14}) are defining relations of the permutation group $S_n$, where the elements $\rho_1,\rho_2,\dots,\rho_{n-1}$ generate the permutation group $S_n$ in $VB_n$, hence the virtual braid group $VB_n$ is generated by the braid group $B_n$ and the permutation group $S_n$. Relations (\ref{eq15})-(\ref{eq16}) are called the mixed relations of the virtual braid group. Note that relations (\ref{eq16}) are equivalent to the  relations
$$
\rho_{i+1} \, \rho_{i}  \, \sigma_{i+1} = \sigma_{i} \, \rho_{i+1}  \, \rho_{i}
$$
for $i=1,2,\dots,n-1$. It was shown in  \cite{GPV} that the relations  
\begin{align}
\rho_{i} \, \sigma_{i+1}  \, \sigma_{i} = \sigma_{i+1} \, \sigma_{i} \, \rho_{i+1},&&
\rho_{i+1} \, \sigma_{i}  \, \sigma_{i+1} = \rho_{i} \, \sigma_{i+1} \, \sigma_{i}. \label{eq101}
\end{align}
donot hold in  $VB_n$. Relations (\ref{eq101}) are usually  called \textit{forbidden relations}. 

\textit{The virtual pure braid group} introduced in~\cite{Bar} is the kernel  $VP_n = \operatorname{Ker}(\nu)$  of the homomorphism  $\nu : VB_n \to S_n$  from the virtual braid group $VB_n$ onto the symmetric group $S_n$ given on generators $\sigma_1,\sigma_2,\dots,\sigma_{n-1},\rho_1,\rho_2,\dots,\rho_{n-1}$ of $VB_n$ by the following rule
$$
\nu(\sigma_i) = \nu(\rho_i) = \rho_i, \qquad  \text{for}~i = 1, 2, \ldots, n-1.
$$
From the definition it is clear that the following short exact sequence holds:
$$
1 \to VP_{n} \to VB_{n} \stackrel{\nu}{\to} S_{n} \to  1.
$$ 
Following~\cite{Bar} for $1\leq i\neq j \leq n$ let us definte the elements $\lambda_{i,j}$ by the following formulas 
\begin{align*}
\lambda_{i,i+1} &= \rho_i \, \sigma_i^{-1},\\
  \lambda_{i+1,i} &= \rho_i \, \lambda_{i,i+1} \, \rho_i = \sigma_i^{-1} \, \rho_i
\end{align*}
for  $i=1, 2, \ldots, n-1$, and
\begin{align*}
\lambda_{i,j} & = \rho_{j-1} \, \rho_{j-2} \ldots \rho_{i+1} \, \lambda_{i,i+1} \, \rho_{i+1} \ldots \rho_{j-2} \, \rho_{j-1}, \\ 
\lambda_{j,i} & =  \rho_{j-1} \, \rho_{j-2} \ldots \rho_{i+1} \, \lambda_{i+1,i} \, \rho_{i+1} \ldots \rho_{j-2} \, \rho_{j-1}
\end{align*} 
for $1 \leq i < j-1 \leq n-1$. The following theorem is proved in \cite{Bar}. 
\begin{theorem}   \label{theorem1.1} 
The group $VP_n$ admits a presentation with the generators $\lambda_{k,l}$ for $1 \leq k \neq l \leq n$ and the defining relations
\begin{align}
\label{eq17}\lambda_{i,j} \,  \lambda_{k,\, l} &= \lambda_{k,\, l}  \, \lambda_{i,j},\\ 
\label{eq18}\lambda_{k,i} \,  \lambda_{k,j} \,  \lambda_{i,j} &= \lambda_{i,j} \,  \lambda_{k,j} \,  \lambda_{k,i}, 
\end{align}
where different letters denote different indices.
\end{theorem}

The \emph{welded braid group} $WB_n$ on $n$ strands introduced in~\cite{FRR}  is the group with the generators $\sigma_1,\sigma_2,\dots,\sigma_{n-1}$, $\alpha_1,\alpha_2,\dots,\alpha_{n-1}$, where the elements $\sigma_1,\sigma_2,\dots,\sigma_{n-1}$ generate the braid group $B_n$ on $n$ strands (i.~e. satisfy relations (\ref{eq1})-(\ref{eq2})), the elements $\alpha_1,\alpha_2,\dots,\alpha_{n-1}$ generate the symmetric group $S_n$ (i.~e. satisfy relations (\ref{eq12})-(\ref{eq14}) where we change the letter $\rho$ by the letter $\alpha$), and moreover, the following mixed relations hold: 
\begin{align}
\label{eq19}\alpha_{i}  \, \sigma_{j} &= \sigma_{j} \, \alpha_{i},&& |i - j| \geq 2,\\
\label{eq20}\sigma_{i} \, \alpha_{i+1} \, \alpha_i &= \alpha_{i+1} \, \alpha_{i} \, \sigma_{i+1},&&  i=1,2,\ldots,n-2,\\
 \label{eq21}\sigma_{i+1} \, \sigma_{i} \, \alpha_{i+1} &= \alpha_{i} \, \sigma_{i+1} \, \sigma_{i}, &&  i=1,2,\ldots,n-2.
\end{align}
It was observed in~\cite{FRR} that $WB_n$ is isomorphic to the group of conjugating automorphisms $C_n$. It was proved in \cite{FRR} that the following relations symmetric to (\ref{eq20}) 
$$
\sigma_{i+1} \, \alpha_{i} \, \alpha_{i+1} = \alpha_{i} \, \alpha_{i+1} \, \sigma_{i},
$$
hold in $WB_n$ for $i=1,2,\dots,n-2$, while the relations 
$$
\alpha_{i+1}  \, \sigma_{i} \, \sigma_{i+1} = \sigma_{i} \, \sigma_{i+1} \, \alpha_{i}
$$
do not hold. 

Comparing defining relations (\ref{eq15})--(\ref{eq16}) of $VB_n$ with defining relations (\ref{eq19})--(\ref{eq21}) of $WB_n$, one can see that the difference between $WB_n$ and  $VB_n$ is given by relations (\ref{eq21}). Therefore, there exists a homomorphism
$$
\varphi_{VW} : VB_n \longrightarrow WB_n,
$$
which maps $\sigma_i$ to $\sigma_i$, and $\rho_i$ to $\alpha_i$ for  $i=1,2\dots,n-1$. Hence, $WB_n$ is the homomorphic image of $VB_n$. 

Recall that a group $G$ is said to be \textit{linear} if there exists a faithful representation of $G$ into the general linear group ${\rm GL}_m(k)$ for some integer $m\geq 2$ and a field $k$. It was shown in \cite{Big, Kra} that the braid group $B_n$ is linear for every $n \geq 2$. 
In \cite{Ver01} there were constructed linear representations of  the virtual braid groups $VB_n$ and the welded braid groups $WB_n$ by matrices from $\mbox{GL}_n(\mathbb{Z}[t, t^{-1}])$. These representations extend  the well-known Burau representation. In \cite{BarP} there was constructed a linear representation of the group $C_n \simeq WB_n$ which is an extension (with additional conditions on parameters) of the well-known  Lawrence-Bigelow-Krammer representation.

\subsection{Braid groups and crystallographic groups} An interesting connection between braid groups and crystallographic groups was observed in~\cite{GGO}.  Recall that a discrete and uniform subgroup $G$ of $\mathbb{R}^n \rtimes O_n (\mathbb{R})\subseteq \operatorname{Aff} (\mathbb{R}^n)$ is said to be a \textit{crystallographic group} of dimension $n$. If in addition $G$ is torsion free, then $G$ is called a \textit{Bieberbach group} of dimension $n$. The correspondence between Bieberbach groups and fundamental groups of closed flat Riemannian manifold is presented in~\cite{Wolf}.

The following characterization of crystallographic groups was used in~\cite[Lemma 8]{GGO} in order to demonstrate the relation between braid groups and crystallographic groups. 

\begin{lemma}\label{lemma1.1} Let $G$ be a group. Then $G$ is a crystallographic group if and only if there exist an integer $n \in  \mathbb{N}$ and a short exact sequence 
$$0 \to \mathbb{Z}^n \to G \to \Phi \to  1$$ 
such that the following conditions hold:
\begin{enumerate}
\item[(a)] $\Phi$ is a finite group,
\item[(b)] the integral representation $\theta : \Phi \to  \operatorname{Aut} (\mathbb{Z}^n)$, induced by conjugation on $\mathbb{Z}^n$ and
defined by $\theta(\varphi)(x) = \pi x \pi^{-1}$, where $x \in \mathbb{Z}^n$, $\varphi \in \Phi$ and $\pi \in G$ is such that $\zeta(\pi) = \varphi$,
is faithful.
\end{enumerate}
\end{lemma} 
\noindent If $G$ is a crystallographic group, then the integer $n$ from Lemma~\ref{lemma1.1}  is called the dimension of $G$, the finite group $\Phi$ is called the  holonomy group of $G$, and the integral representation $\theta : \Phi \to  \operatorname{Aut} (\mathbb{Z}^n)$ is called the  holonomy representation of $G$.

Let $P'_{n} = [P_{n}, P_{n}]$ be the commutator subgroup of the pure braid group $P_{n}$. It was shown in~\cite{GGO} that the quotient $B_n /P_n'$ is a crystallographic group for all $n \geq 3$. Moreover, it was proved that $B_n /P_n'$ possesses torsion, and  there is
a one-to-one correspondence between the conjugacy classes of the finite-order elements of $B_n /P_n'$ and the conjugacy classes of the elements of odd order of the symmetric group $S_n$.

The groups $B_n /P_n'$ also arise in the study of pseudo-symmetric braided categories~\cite{PS}. In this paper the authors  consider the quotient, denoted by $PS_n$, of $B_n$ by the normal subgroup generated by the relations
$$
s_i s^{-1}_{i+1} s_i = s_{i+1} s^{-1}_i s_{i+1}, \qquad  i = 1,  \ldots, n-2,
$$
and show that $PS_n$ is isomorphic to $B_n /P_n'$.

\section{Linear representation of $B_n/P_n'$}	\label{sec2} 

Let $R = \Z[q^{\pm 1}, t^{\pm 1}]$ be the ring of Laurent polynomials with two variables $q, t$, and $V$ be the free $R$-module with the basis $e_{i,j}$ for $1\leq i < j \leq n$.
The Lawrence-Bigelow-Krammer representation (shortly LBKR) is the linear representation of the braid group $B_n$ by the automorphisms of $V$
$$
\psi = \psi_{q,t}: B_n \to \operatorname{Aut} (V)
$$
which is given on the generators $\sigma_1,\sigma_2,\dots,\sigma_{n-1}$ of $B_n$ by the following rule
\begin{equation}\label{LKBnnnn} \psi(\sigma_k)(e_{i,j})= \begin{cases}
e_{i,j}, & k < i -1\text{ or } k > j,\\
e_{i-1, j} + (1-q) e_{i,j}, & k = i - 1,\\
tq(q-1) e_{i, i+1} + qe_{i+1, j}, & k = i  <  j-1,\\
tq^2 e_{i, j},  &  k = i  =  j-1,\\
e_{i,j} +  tq^{k-i}(q-1)^2 e_{k, k+1},  &  i < k  <  j-1,\\
e_{i, j-1} + tq^{j-i}(q-1) e_{j-1, i}, &  i< k = j-1,\\
(1-q) e_{i, j} + qe_{i, j+1}, &  k = j.\\
\end{cases}
\end{equation}

\begin{example}{\rm 
If $n=3$, then the braid group $B_3$ is generated by $\sigma_1,\sigma_2$, and the module $V$ has the basis $\{ e_{1,2},  e_{1,3}, e_{2,3}\}$. From formula (\ref{LKBnnnn}) it follows that the linear maps $\psi(\sigma_1)$, $\psi(\sigma_2)$, $\psi(\sigma_1)^{-1}$, $\psi(\sigma_2)^{-1}$ can be presented by the following matrices, which are written in the basis $e_{1,2}, e_{1,3}, e_{2,3}$ of $V$
\begin{align*}
[\psi(\sigma_1)] &=
\begin{pmatrix}
	tq^2 & 0 & 0 \\
	tq(q-1) & 0 & q \\
	0 & 1 & 1-q
\end{pmatrix},\\ 
[\psi(\sigma_2)] &=
\begin{pmatrix}
	1-q & q & 0 \\
	1 & 0 & tq^2(q-1) \\
	0 & 0 & tq^2
\end{pmatrix},\\
[\psi(\sigma_1)]^{-1} &=
\begin{pmatrix}
	t^{-1}q^{-2} & 0 & 0 \\
-1+2q^{-1} -q^{-2} & 1-q^{-1} & 1 \\
-q^{-1}+q^{-2} & q^{-1} &0
\end{pmatrix}, \\
[\psi(\sigma_2)]^{-1} &=
\begin{pmatrix}
0 & 1 & -q+1 \\
q^{-1}  & 1-q^{-1} & -q+2-q^{-1} \\
0 & 0 & 	t^{-1}q^{-2}
\end{pmatrix}.
\end{align*}}
\end{example} 

For two complex numbers $q_0, t_0\in\mathbb{C}$ denote by $V_{q_0, t_0}$ the free module  over the ring $\mathbb{Z}[q_0, t_0]$ with the same basis as $V$, and denote by $\psi_{q_0, t_0}$ the specialization of $\psi_{q, t}$ obtained by putting $q=q_0$, $t=t_0$.  The above specialization gives the linear representation $\psi_{q_0, t_0} : B_n \to \operatorname{Aut} (V_{q_0, t_0})$. In particular, the map $\psi_{1,t} : B_n \to \operatorname{Aut} (V_{1,t})$ is defined on the generators $\sigma_1,\sigma_2,\dots,\sigma_{n-1}$ of $B_n$ by the following rules
$$
\psi_{1,t}(\sigma_k) :\begin{cases}
e_{k,k+1} \mapsto  t e_{k,k+1}, &  \\
e_{i,j} \mapsto  e_{s_k(i),s_k(j)}, & \text{if}~(i,j) \not= (k,k+1)~\text{and}~1\leq i < j \leq n,\\
\end{cases}
$$
where $s_{k} = (k,k+1)$ is the transposition from $S_{n}$. In this formula we assume that $e_{j,i}=e_{i,j}$ if $i<j$.
\begin{lemma} \label{lemma2.1}
If $q = 1$ and $t = \pm1$, then the image of $B_n$ under the specialization $\psi_{1, \pm1}$ of LBKR is isomorphic
to  the symmetric group $S_n$.
\end{lemma}
\begin{proof}
It is enough to observe that $\left(\psi_{1, \pm1}(\sigma_i)\right)^2$ is the identity automorphism. This property immidiately follows from the definition of $\psi_{1, \pm1}$.
\end{proof}
From the expressions of the generators $a_{i,j}~(1\leq i<j\leq n)$ of $P_n$ introduced in Section~\ref{braidsandautomorphisms} and direct calculations we have the following result. 
\begin{lemma} \label{lemma2.2}
The representation $\psi_{1,t}:B_n\to{\rm Aut}(V_{1,t})$ acts on the generators  $a_{i,j}~(1\leq i<j\leq n)$ of $P_n$ by the following rules
$$
\psi_{1,t}(a_{ij}) : \begin{cases}
e_{i,j} \mapsto t^2 e_{i,j}, &  \\
e_{k,l} \mapsto e_{k,l}, & ~\text{for}~ (i,j) \not= (k, l).\\
\end{cases}
$$
\end{lemma}
The following theorem is the main results of the present section.
\begin{theorem} \label{theorem2.1}
The image of $B_n$ under  $\psi_{1,t}$ is isomorphic to  $B_n/P_n'$.
\end{theorem}
\begin{proof}In Section~\ref{braidsandautomorphisms} we have noticed that the following short exact sequence holds
	$$1\to P_n\to B_n\to S_n \stackrel{\pi}{\to} 1.$$
From this short exact sequence and the isomorphism $$(B_n/P_n^{\prime})/(P_n/P_n^{\prime})=B_n/P_n$$
we have the following short exact sequence
$$
1 \to P_n/P_n' \to B_n/P_n' \stackrel{\pi}{\to}  S_n \to 1.
$$
From this short exact sequence it follows that every element $b \in B_n/P_n'$ can be expressend in the form $b = p \lambda$, where $p \in P_n / P_n'$ and $\lambda$ is a coset representative of $P_n / P_n'$ in $B_n / P_n'$. In order to prove the statement of the theorem we have to prove that if $b$ is non-trivial, then $\psi_{1,t}(b)$ in non-identity automorphism. 

Let $b=p\lambda\neq 1$. If $\lambda \not= 1$, then $\pi(b) \not= 1$, and hence the composition of $\psi_{1,t}$ and specialization $t=1$ send $b$ to non-trivial automorphism. Therefore,  $\psi_{1,t}(b)$ is non-trivial. If $\lambda = 1$ and $b \not= 1$, then $b$ can be presented in the form
$$
b = a_{1,2}^{\alpha_{1,2}} \left(a_{1,3}^{\alpha_{1,3}}  a_{2,3}^{\alpha_{2,3}}\right)  \ldots \left(a_{1,n}^{\alpha_{1,n}} \ldots a_{n-1,n}^{\alpha_{n-1,n}}\right) P_n'
$$
for some integers $\alpha_{i,j} \in \mathbb{Z}$. From Lemma~\ref{lemma2.2} it follows that the matrix of the automorphism  $\psi_{1,t}(b)$ is the following diagonal matrix
$$
[\psi_{1,t}(b)] = \operatorname{diag} (t^{2\alpha_{1,2}}, t^{2\alpha_{1,3}}, t^{2\alpha_{2,3}} \ldots t^{2\alpha_{1,n}} \ldots t^{2\alpha_{n-1,n}}),
$$
which is not the identity matrix. Hence, we have proved that if $b$ is non-trivial, then its image $\psi_{1,t}$ is non-trivial.
\end{proof}

It was shown in~\cite{GGO} that the group $B_n / P_n'$ possesses torsion, and there
is a one-to-one correspondence between the conjugacy classes of the finite-order elements of
$B_n / P_n'$ and the conjugacy classes of the elements of odd order of the symmetric group
$S_n$. The group $B_3/P_3'$ was considered separately in the paper~\cite{GGO}. Using the linear representation $\psi_{1,t}$ we obtain the following statement.
\begin{proposition} \label{proposition2.1} 
The group $B_3/P_3'$ has only $3$-torsion. Every element which has one of the following forms
\begin{align*}
a_{1,2}^{\alpha} a_{1,3}^{\beta} a_{2,3}^{\gamma} \sigma_1 \sigma_2,&& a_{1,2}^{\alpha} a_{1,3}^{\beta} a_{2,3}^{\gamma} \sigma_2 \sigma_1,
\end{align*}
where $\alpha, \beta, \gamma \in \mathbb{Z}$ are such that $\alpha+\beta+\gamma=-1$, has order $3$. Any other non-trivial element has infinite order.
\end{proposition}

\begin{proof}
Since $\psi_{1,t}(B_n) \cong B_3/P_3'$, we can present elements of $B_3/P_3'$ by matrices in $\operatorname{GL}_{3}(\mathbb{Z}[t^{\pm 1}])$. For $n=3$ any element $b \in B_3/P_3'$ can be presented in the form $b = p \lambda$, where $p \in P_3/P_3'$ has the following form
$$
p = \operatorname{diag} (t^{2\alpha}, t^{2\beta}, t^{2\gamma}), \qquad ~\text{for}~\alpha, \beta, \gamma \in \mathbb{Z},
$$
and $\lambda$ is one of the following matrices
\begin{align*}
\lambda_0 &= 
\begin{pmatrix}
	1 & 0 & 0 \\
	0 & 1 & 0 \\
	0 & 0 & 1
\end{pmatrix},&&&\lambda_1 =
\begin{pmatrix}
	t & 0 & 0 \\
	0 & 0 & 1 \\
	0 & 1 & 0
\end{pmatrix},&& \lambda_2 =
\begin{pmatrix}
	0 & 1 & 0 \\
	1 & 0 & 0 \\
	0 & 0 & t
\end{pmatrix},\\
\lambda_3 &=
\begin{pmatrix}
	0 & t & 0 \\
	0 & 0 & t \\
	1 & 0 & 0
\end{pmatrix},&&& 
\lambda_4 =
\begin{pmatrix}
	0 & 0 & 1 \\
	t & 0 & 0 \\
	0 & t & 0
\end{pmatrix},&& 
\lambda_5 =
\begin{pmatrix}
	0 & 0 & t \\
	0 & t & 0 \\
	t & 0 & 0
\end{pmatrix}.
\end{align*}
Let us consider each $\lambda_i$ separately. First, it is obvious that the element $p\lambda_0$ has infinite order. If the element $p\lambda_1$ has a finite order, then from the matrix form of $\lambda_1$ it is clear that this order has to be even. Since for an arbitrary integer $k$ we have
$$
(p\lambda_1)^{2k} = \operatorname{diag} (t^{2k(2\alpha+1)},  t^{2k(\beta+\gamma)},  t^{2k(\beta+\gamma)}),
$$
the element $p\lambda_1$ has infinite order.
Further if the element $p\lambda_2$ has a finite order, then from the matrix form of $\lambda_2$ we see that this order is even. Since for an arbitrary integer $k$ we have
$$
(p\lambda_2)^{2k} = \operatorname{diag} (t^{2k(\alpha+\beta)},  t^{2k(\alpha+\beta)},  t^{2k(2\gamma+1)}),
$$
the element $p\lambda_2$ has infinite order. From the matrix representations it is clear that the order of the element $p\lambda_3$ is divisible by $3$. For an arbitrary $k$ we have the equality
$$
(p\lambda_3)^{3k} = \operatorname{diag} (t^{2k(\alpha+\beta+\gamma+1)},  t^{2k(\alpha+\beta+\gamma+1)},  t^{2k(\alpha+\beta+\gamma+1)}),
$$
i.~e. if $\alpha+\beta+\gamma+1 =0$, then  $p\lambda_3$ has order 3, otherwise $p\lambda_3$ has infinite order. Using the same argument, the order of the element $p\lambda_4$ is divisible by $3$, and for an arbitrary $k$ we have
$$
(p\lambda_4)^{3k} = \operatorname{diag} (t^{2k(\alpha+\beta+\gamma+1)},  t^{2k(\alpha+\beta+\gamma+1)},  t^{2k(\alpha+\beta+\gamma+1)}),
$$
i.~e. if $\alpha+\beta+\gamma+1 =0$, then  $p\lambda_4$ has order 3, otherwise $p\lambda_4$ has infinite order. It is easy to check that elements of the form $p\lambda_5$ have infinite order
\end{proof}
In general case it is not difficult to prove the following statement.
\begin{proposition} \label{proposition2.2} 
If $b$ is a non-trivial element in $B_n/P_n'$ such that the permutation $\pi(b)$ is non-identity but fixes some points, then $b$ has infinite order.
\end{proposition}

\section{Linear representation of $VB_n / VP_n'$} \label{sec3} 
In Section~\ref{virtbraid} we have noticed that the following short exact sequence holds
$$1\to VP_n\to VB_n\to S_n \to 1.$$
From this short exact sequence and the isomorphism $$(VB_n/VP_n^{\prime})/(VP_n/VP_n^{\prime})=VB_n/VP_n$$
we have the following short exact sequence
$$
1 \to VP_n/VP_n' \to VB_n/VP_n'\to  S_n \to 1,
$$
where $VP_n / VP_n'  = \mathbb{Z}^{n(n-1)}$ (since $VP_n$ is generated by the elements $\lambda_{i,j}$ for $1\leq i\neq j \leq n$), $S_n = \langle \rho_1, \rho_2, \ldots, \rho_{n-1} \rangle$, and the group $S_n$ acts on the generators of $VP_n / VP_n' $ by the permutation of indices. From the equality $VB_n=VP_n\rtimes S_n$ it follows that $VB_n / VP_n' =  (VP_n / VP_n') \rtimes S_n$, and hence, $VB_n / VP_n'$ is a crystallographic group.

In this section we construct an  extension of the representation $\psi_{1,t}$ from the braid group $B_n$ to the virtual braid group $VB_n$. Let $U$ be  the free module with the basis $e_{i,j}$, $1 \leq i<j\leq n$ over the ring $\mathbb{Z}[t^{\pm1}, t_1^{\pm1},t_2^{\pm1},\ldots, t_{n-1}^{\pm1}]$ of Laurent polinomials with $n$ variables $t,t_1,\dots,t_{n-1}$. Let us define the map $\Psi: VB_n \to \operatorname{Aut} (U)$ on the generators $\sigma_1,\sigma_2,\dots,\sigma_{n-1}$, $\rho_1,\rho_2,\dots,\rho_{n-1}$ of $VB_n$ by the following rules:
\begin{align}
\label{eq102}\Psi(\sigma_i) &= \psi_{1,t}(\sigma_i), \\
\label{eq103}\Psi(\rho_{i}) &: \begin{cases}
e_{i,i+1} \mapsto e_{i,i+1}, &  \\
e_{k,l} \mapsto e_{k,l}, & k\not=i,i+1, l\not=i,i+1,\\
e_{i,j} \mapsto t_i e_{i+1,j}, & i<j-1,\\
e_{i+1,j} \mapsto t_i^{-1} e_{i,j}, &\\
e_{k,i} \mapsto t_i e_{k, i+1}, &\\
e_{k,i+1} \mapsto t_i^{-1} e_{k,i}. &
\end{cases}
\end{align} 

The following property can be easy verified using direct calculations. 
\begin{proposition} \label{proposition3.1}
The map $\Psi : VB_n \to \operatorname{Aut} (U)$, defined by (\ref{eq102}) and (\ref{eq103}) is a representation. The images $\Psi (\sigma_{i})$ and $\Psi(\rho_{i})$ satisfy both forbidden relations (\ref{eq101}) for $VB_n$.
\end{proposition}

The following statement describes the image of $\Psi$.
\begin{theorem} \label{theorem3.1} 
The image $\Psi(VB_n)$ is isomorphic to $VB_n/VP_n'$.
\end{theorem}
\begin{proof} Using direct calculations it is easy to see that the image $\Psi(b)$ of every element $b\in VP_n'$ is the trivial automorphism. Hence, in order to prove the theorem we have to prove that the  image of any elements $b \in VB_n \setminus VP_n'$ is non-trivial. If $b$ does not belong to $VP_n$, then the image of $b$ under $\Psi$ is clearly non-trivial automorphism of $U$. If $b$ is a non-trivial element of $VP_n$, then looking to the representative of $b$ in $bVP_n'$ it is easy to see that the matrix of $\Psi(b)$ is a non-identity diagonal matrix.
\end{proof}

The following questions arise naturally from the paper \cite{GGO} where the same questions are studied for classical linear groups. 

\begin{question}
What are the possible finite orders of elements from $VB_n/VP_n'$ which do not belong to $S_n$? 
\end{question} 
\begin{question}  
Which Bieberbach subgroups belong to $VB_n/VP_n'$?
\end{question}
Let us consider the particular case $n=3$. 
The virtual braid group $VB_3$  is generated by four elements $\sigma_1, \sigma_2, \rho_1, \rho_2$ and it has the following defining relations
\begin{align*}
 \sigma_1 \sigma_2 \sigma_1 = \sigma_2 \sigma_1 \sigma_2, &&  \rho_1 \rho_2 \rho_1 = \rho _2 \rho_1 \rho _2,  && \rho_1^2 = \rho_2^2 = 1, && \rho_1 \rho_2 \sigma_1 =\sigma_2 \rho_1 \rho_2.
\end{align*}
The images of the generators $\sigma_1,\sigma_2,\rho_1,\rho_2$ of $VB_n$ under the representation $\Psi : VB_3 \to \operatorname{GL}_3(\mathbb{Z}[t^{\pm1}, t_1^{\pm1},t_2^{\pm1}]$ are the following matrices
\begin{align*}
\Psi(\sigma_1) &=
\begin{pmatrix}
	t & 0 & 0 \\
	0 & 0 & 1 \\
	0 & 1 & 0
\end{pmatrix},&& \Psi(\sigma_2) =
\begin{pmatrix}
	0 & 1 & 0 \\
	1 & 0 & 0 \\
	0 & 0 & t
\end{pmatrix},\\
\Psi(\rho_1) &=
\begin{pmatrix}
	1 & 0 & 0 \\
	0 & 0 & t_1 \\
	0 & t_1^{-1} & 0
\end{pmatrix},&& \Psi(\rho_2) =
\begin{pmatrix}
	0 & t_2 & 0 \\
	t_2^{-1} & 0 & 0 \\
	0 & 0 & 1
\end{pmatrix}.
\end{align*}
The group $VP_3$ is generated by the elements 
\begin{align*} 
&\lambda_{12} = \rho_1 \sigma_1^{-1},&& \lambda_{21} = \sigma_1^{-1} \rho_1,&&\lambda_{13} = \rho_2 \lambda_{12} \rho_2, \\ 
&\lambda_{31} = \rho_2 \lambda_{21} \rho_2,&& \lambda_{23} = \rho_2 \sigma_2^{-1},&& \lambda_{32} = \sigma_2^{-1} \rho_2.
\end{align*} 
Using direct calculations it is easy to see that the images of the elements $\lambda_{12}$, $\lambda_{21}$, $\lambda_{13}$, $\lambda_{31}$, $\lambda_{23}$, $\lambda_{32}$  under the representations $\Psi$ are the following diagonal matrices
\begin{align*} 
&\Psi(\lambda_{12}) = \operatorname{diag}(t^{-1}, t_1, t_1^{-1}),&& \Psi(\lambda_{21}) = \operatorname{diag} (t^{-1},  t_1^{-1}, t_1), \\ 
&\Psi(\lambda_{13}) = \operatorname{diag}(t_1, t, t_1^{-1}),&& \Psi(\lambda_{31}) = \operatorname{diag} (t_1^{-1},  t^{-1}, t_1), \\ 
&\Psi(\lambda_{23}) = \operatorname{diag}(t_2, t_2^{-1}, t^{-1}),&& \Psi(\lambda_{32}) = \operatorname{diag} (t_2^{-1}, t_2,  t^{-1}), 
\end{align*} 
in particular, $\Psi(\lambda_{12}) \not =\Psi(\lambda_{21})$. Let us demonstrate that there are torsion elements in $VB_3/VP_3'$ which do not belong to $S_3$. Consider the element  
$$
p = \lambda_{12}^{a} \lambda_{21}^{\alpha} \lambda_{13}^b \lambda_{23}^c \lambda_{31}^{\beta} \lambda_{32}^{\gamma} d \rho_1, 
$$
where $\alpha, \beta, \gamma, a, b, c \in \mathbb{Z}$ and $d \in VP_3'$. 
Then
$$
\Psi(p)^2 = \operatorname{diag} \left(t^{-2(a+\alpha)} t_1^{2(b-\beta)} t_2^{-2(c-\gamma)}, t^{b-\beta -\gamma-c} t_1^{\beta-b} t_2^{\gamma-c}, t^{b-\beta -\gamma-c} t_1^{\beta-b} t_2^{2(\gamma-c)} \right).
$$
From this formula it is clear that if $\alpha = -a$, $\beta = b$, and $\gamma=c=0$, then $\Psi(p)^2 = E$. 
The same arguments lead to the following statement. 

\begin{proposition} \label{proposition3.2}
The group $VB_3 / VP_3'$ has infinitely many  elements of order 2.
\end{proposition}

\section{Decomposition and linearity  of $VB_3$} \label{sec4} 
Properties of the group $VP_3$ were studied in~\cite{BMVW}. In particular, the certain decomposition of this group was found in~\cite{BMVW}. In the present section we construct a simpler decomposition and prove that $VP_3$, and hence $VB_3$, is linear. The following statement is the first main result of this section.
\begin{proposition} \label{proposition4.1}  
$VP_3= (F_3\rtimes F_2)* \mathbb Z$. 
\end{proposition} 
\begin{proof}
By Theorem~\ref{theorem1.1}, the group  $VP_3$ has a presentation with six generators $\lambda_{12}$, $\lambda_{21}$, $\lambda_{13}$, $\lambda_{23}$, $\lambda_{32}$, $\lambda_{31}$, and six relations which can be write as the following conjugations: 
\begin{eqnarray}
\lambda_{12}(\lambda_{13} \lambda_{23})=(\lambda_{23}\lambda_{13})\lambda_{12} & \Leftrightarrow & (\lambda_{13}\lambda_{23})^{\lambda_{12}^{-1}}=\lambda_{23}\lambda_{13},\label{reln1}\\
\lambda_{13}(\lambda_{12} \lambda_{32})=(\lambda_{32}\lambda_{12})\lambda_{13} & \Leftrightarrow & (\lambda_{32}\lambda_{13}^{-1})^{\lambda_{12}^{-1}}=\lambda_{13}^{-1}\lambda_{32},\label{reln2}\\
\lambda_{31}(\lambda_{32} \lambda_{12})=(\lambda_{12}\lambda_{32})\lambda_{31} & \Leftrightarrow & (\lambda_{32}\lambda_{31}^{-1})^{\lambda_{12}^{-1}}=\lambda_{31}^{-1}\lambda_{32},\label{reln3}\\
\lambda_{21}(\lambda_{23} \lambda_{13})=(\lambda_{13}\lambda_{23})\lambda_{21} & \Leftrightarrow & (\lambda_{23}\lambda_{13})^{\lambda_{21}^{-1}}=\lambda_{13}\lambda_{23},\label{reln4}\\
\lambda_{23}(\lambda_{21} \lambda_{31})=(\lambda_{31}\lambda_{21})\lambda_{23} & \Leftrightarrow & (\lambda_{31}\lambda_{23}^{-1})^{\lambda_{21}^{-1}}=\lambda_{23}^{-1}\lambda_{31},\label{reln5}\\
\lambda_{32}(\lambda_{31} \lambda_{21})=(\lambda_{21}\lambda_{31})\lambda_{32} & \Leftrightarrow & (\lambda_{31}\lambda_{32})^{\lambda_{21}^{-1}}=\lambda_{32}\lambda_{31}.\label{reln6}
\end{eqnarray}
Note that equations (\ref{reln1}) and (\ref{reln2}) are also equivalent to the following commutation relations
\begin{align*}
(\lambda_{13}\lambda_{23})^{\lambda_{12}^{-1}\lambda_{13}^{-1}}=\lambda_{13}\lambda_{23},&& (\lambda_{32}\lambda_{13}^{-1})^{\lambda_{12}^{-1}\lambda_{13}^{-1}}=\lambda_{32}\lambda_{13}^{-1}, 
\end{align*}
respectively. Let us define the following elements of $VP_{3}$:
\begin{align*}
\boo=\lambda_{21}\lambda_{12},&& \bot=\lambda_{13}\lambda_{12},&& \bto=\lambda_{13}\lambda_{23},&& \btt=\lambda_{32}\lambda_{31},&& \bc=\lambda_{32}\lambda_{23}, 
\end{align*}
and rewrite all relations (\ref{reln1})-(\ref{reln6}) in order to express how the group $\langle b_{11},b_{12}\rangle$ acts on the group $\langle b_{21}, b_{22}, b_{23}\rangle$ by conjugations:
\begin{align*}
\bto^{\bot^{-1}}  &= (\lambda_{13}\lambda_{23})^{\lambda_{12}^{-1}\lambda_{13}^{-1}} \stackrel{(\ref{reln1})}{=}  \lambda_{13}\lambda_{23} = \bto,\\
\btt^{\bot^{-1}} & =  \lambda_{13}\lambda_{12}\cdot\lambda_{32}\lambda_{31}\cdot\lambda_{12}^{-1}\lambda_{13}^{-1}
  \stackrel{(\ref{reln3})}{=}  \lambda_{13}\cdot\lambda_{31}\lambda_{32}\cdot\lambda_{13}^{-1}\\
 & =  \lambda_{13}\lambda_{23}\cdot \lambda_{23}^{-1}\lambda_{32}^{-1}\cdot \lambda_{32}\lambda_{31}\cdot \lambda_{32}\lambda_{23}\cdot \lambda_{23}^{-1}\lambda_{13}^{-1} =  \btt^{\bc \bto^{-1}},\\
\bc^{\bot^{-1}} & =  (\lambda_{32}\lambda_{23})^{\lambda_{12}^{-1}\lambda_{13}^{-1}} = (\lambda_{32}\lambda_{13}^{-1}\cdot\lambda_{13}\lambda_{23})^{\lambda_{12}^{-1}\lambda_{13}^{-1}}\\ 
  &\stackrel{(\ref{reln1}),(\ref{reln2})}{=}  \lambda_{32}\lambda_{13}^{-1}\cdot\lambda_{13}\lambda_{23} = \bc,\\
\bto^{\boo^{-1}}  &=  (\lambda_{13}\lambda_{23})^{\lambda_{12}^{-1}\lambda_{21}^{-1}} \stackrel{(\ref{reln1}),(\ref{reln4})}{=}  \lambda_{13}\lambda_{23}=\bto,\\
\btt^{\boo^{-1}}&=(\lambda_{32}\lambda_{31})^{\lambda_{12}^{-1}\lambda_{21}^{-1}} \stackrel{(\ref{reln3}),(\ref{reln6})}{=} \lambda_{32}\lambda_{31}=\btt,\\
\bc^{\boo^{-1}} & =  (\lambda_{32}\lambda_{13}^{-1}\cdot\lambda_{13}\lambda_{23})^{\lambda_{12}^{-1}\lambda_{21}^{-1}}
  \stackrel{(\ref{reln2}),(\ref{reln1})}{=}  (\lambda_{13}^{-1}\lambda_{32}\cdot\lambda_{23}\lambda_{13})^{\lambda_{21}^{-1}}\\
 & =  (\lambda_{13}^{-1}\lambda_{23}^{-1}\cdot \lambda_{23}\lambda_{31}^{-1}    \cdot \lambda_{31}\lambda_{32}\cdot\lambda_{23}\lambda_{13})^{\lambda_{21}^{-1}}\\
  & \stackrel{(\ref{reln4}),(\ref{reln5}),(\ref{reln6}),(\ref{reln4})}{=}  \lambda_{23}^{-1}\lambda_{13}^{-1}\cdot \lambda_{31}^{-1}\lambda_{23} \cdot \lambda_{32}\lambda_{31}\cdot\lambda_{13}\lambda_{23}\\
 & =  \lambda_{23}^{-1}\lambda_{13}^{-1}\cdot\lambda_{31}^{-1}\lambda_{32}^{-1}\cdot\lambda_{32}\lambda_{23}\cdot \lambda_{32}\lambda_{31}\cdot \lambda_{13}\lambda_{23}   =  \bc^{\btt\bto}.
\end{align*}
To summarize, we have six calculations which show the following six relations:
\begin{align}
\label{newr1}&\bto^{\bot^{-1}}=\bto, &&\btt^{\bot^{-1}}=\btt^{\bc\bto^{-1}},&&  \bc^{\bot^{-1}}=\bc,\\
\label{newr2}&\bto^{\boo^{-1}}=\bto,&& \btt^{\boo^{-1}}=\btt,&& \bc^{\boo^{-1}}=\bc^{\btt\bto}.
\end{align}
Moreover, in the sequence of six calculations we made, we used exactly one of the relations (\ref{reln1})-(\ref{reln6}) which was not yet used in the preceding ones. It means that the set of relations (\ref{newr1})-(\ref{newr2}) is equivalent to the set of relations (\ref{reln1})-(\ref{reln6}). In particular, the groups $F_2=\langle \boo,\bot\rangle$, $F_3=\langle \bto, \btt, \bc\rangle$ are free, and the five elements $b_{11},b_{12}, b_{21}, b_{22}, b_{23}$ generate the  semidirect product $F_3\rtimes F_2$, where $F_2$ acts on $F_3$ by the way described in (\ref{reln1})-(\ref{reln6}). Finally, the five generators $b_{11},b_{12}, b_{21}, b_{22}, b_{23}$ together with any one of the generators $\lambda_{ij}$, for instance with $\lambda_{13}$, generate the full group $VP_3$. Since the six relations describing the semidirect product structure form a complete set of relations of $VP_3$, the group $\langle \lambda_{13}\rangle$ forms a free factor in $VP_3$ and we have the decomposition
$$
VP_3=(\langle \bto, \btt, \bc\rangle \rtimes \langle \boo,\bot\rangle)*\langle \lambda_{13}\rangle.
$$
The proposition is proved.
\end{proof}

Let us now prove that the group $VP_3$ is linear. Denote by $\mu: VB_n \to S_n$ the homomorphism defined on the generators $\sigma_1,\sigma_2,\dots,\sigma_{n-1},\rho_1,\rho_2,\dots,\rho_{n-1}$ of $VB_n$ as follows:
\begin{align*}
\mu(\sigma_i)=1,&& \mu(\rho_i)=\rho_i,&& i=1,2, \dots, n-1\, ,
\end{align*}
where $S_n=\langle \rho_1, \rho_2, \ldots, \rho_{n-1} \rangle$. Denote by $H_n$ the normal closure of $B_n$ in $VB_n$.
It is obvious that $\operatorname{Ker} (\mu)$ coincides with $H_n$.
Now, define the following elements:
\begin{align*}
x_{i,i+1}=\sigma_i,&& x_{i+1,i}=\rho_i \sigma_i \rho_i,
\end{align*}
and 
\begin{align*}
x_{i,j}&=\rho_{j-1} \cdots \rho_{i+1} \sigma_i \rho_{i+1} \cdots \rho_{j-1}&& \mbox{for} \quad 1 \le i < j-1 \le n-1,\\
x_{j,i}&=\rho_{j-1} \cdots \rho_{i+1} \rho_i \sigma_i \rho_i \rho_{i+1} \cdots \rho_{j-1}&& \mbox{for} \quad 1 \le i < j-1 \le n-1.
\end{align*}
The symmetric group  $S_n=\langle \rho_1,\rho_2,\dots, \rho_{n-1} \rangle$
acts on the elements $x_{i,j}$ for $1 \le i \not= j \le n$ by permutation of indices, i.~e.,
$$
\rho x_{i,j} \rho^{-1}=x_{\rho(i),\rho(j)}, \qquad \rho \in S_n.
$$

The following proposition is proved in \cite{R} (see also \cite{BarBel}).

\begin{proposition}\label{prop2}
1) The group $H_n$ admits a presentation with the generators $x_{k,l}$ for $1 \leq k \neq l \leq
n$, and the defining relations:
\begin{align} \label{eq40}
x_{i,j} \,  x_{k,\, l} &= x_{k,\, l}  \, x_{i,j},\\ 
\label{eq41}
x_{i,k} \,  x_{k,j} \,  x_{i,k} &=  x_{k,j} \,  x_{i,k} \, x_{k,j},
\end{align}
where  distinct letters stand for distinct indices.

2) The group  $VB_n$ is isomorphic to  $H_n \rtimes S_n$ where $S_n$ acts on the generators $x_{i,j}$ of $H_n$ by permutation of indices.
\end{proposition}

\begin{corollary}\label{prop:lcsHn}
The group $H_2$ is isomorphic to $\Z*\Z$  which is residually nilpotent.
\end{corollary}

Now we are ready prove the second main result of this section.

\begin{theorem} \label{theorem4.1}
The group $VB_3$ is linear.
\end{theorem}

\begin{proof}
Since $VB_3 = H_3 \rtimes S_3$, the group $VB_3$ is linear if and only if $H_3$ is linear. The group  $H_3$ decomposes into a free product $G_1 * G_2$, where $G_1 = \langle x_{1,2}, x_{2,3}, x_{3,1} \rangle$ and $G_2 = \langle x_{1,3}, x_{3,2}, x_{2,1} \rangle$. The group $G_i,$ $i=1,2$ is isomorphic to the $2$-nd affine Artin-Tits group of type $\mathcal{A}$, also called circular braid group on $3$ strands (see~\cite{AJ}). Since $G_i$ can be embedded in $B_4$ (see~\cite{S}), it is linear. Since the free product of linear groups is linear, $H_3$ is linear. 
\end{proof}

\begin{question}
What is the minimal dimension of the faithful linear representation of $VP_3$?
\end{question}

\begin{question}
What is the image of $H_n$ under the representation $\Psi$?
\end{question}

\section{Universal braid group} \label{sec5} 
The universal braid group $UB_n$ was defined in \cite{Bar} as a group with generators $\sigma_1, \sigma_2, \ldots, \sigma_{n-1},$
$\tau_1, \tau_2, \ldots, \tau_{n-1}$ which satisfy braid relations (\ref{eq1})-(\ref{eq2}), commutativity relations 
$$
\tau_i \, \tau_j = \tau_j \, \tau_i, \qquad |i-j| \geq 2,
$$
and mixed relations 
$$
\tau_i \, \sigma_j = \sigma_j \, \tau_i \qquad |i-j| \geq 2.
$$

Recall~\cite{BS} that \textit{the Artin's group of type $I$}, denoted by $A_I$,  is the group  with generators $a_i$, $i \in I$ and defining relations
$$
a_i \, a_j \, a_i \ldots = a_j \, a_i \, a_j \ldots ,\qquad  i, j \in I,
$$
where the words from the left and from the right hand side of the equality above consist of $m_{ij}$ alternating letters $a_i$ and $a_j$. The following statement is proved in \cite{Bar}.

\begin{proposition}  \label{proposition1}
1) The elements $\sigma_1,\sigma_2,\dots,\sigma_{n-1}$ generate the braid group $B_n$ in the universal braid group $UB_n$.

2) There are homomorphisms $\varphi_{US} : UB_n \to SG_n$,  $\varphi_{UV} : UB_n \to VB_n$ and $\varphi_{UB} : UB_n \to B_n$. 

3) The group $UB_n$ is the Artin's group.
\end{proposition}
Statement 2) of Proposition~\ref{proposition1} is the reason why the group $UB_n$ is called the universal braid group. The commutator subgroup of $UB_n$ is studied in \cite{DG}.

Denote by  $T_n$ the group  generated by $\tau_1$, $\tau_2$, $\ldots$, $\tau_{n-1}$ with the following defining commutativity relations 
$$
\tau_i \tau_j = \tau_j \tau_i, \qquad |i-j| > 1.
$$
It is clear that $T_n$ forms a subgroup in the singular braid group $SG_n$ and a subgroup in the universal braid group $UB_n$. 
By adding to $T_{n}$ the relations $\tau_i^2 = 1$ for $i=1,2, \ldots,n-1$ we get the twin group or the flat braid group (see \cite{BSV}). 

Since the elements $\sigma_1,\sigma_2,\dots,\sigma_{n-1}$ generate a subgroup of $UB_n$ which is isomorphic to the braid group $B_n$, the elements $\tau_1,\tau_2,\dots,\tau_{n-1}$ generate a subgroup of $UB_n$ which is isomorphic to  $T_n$, and the group $UB_n$ is generated by the elements $\sigma_1,\sigma_2,\dots,\sigma_{n-1}$, $\tau_1,\tau_2,\dots,\tau_{n-1}$, in order to define some representation of $UB_n$ we have to define a representation of $B_n<UB_n$, a representation of $T_n<UB_n$, and then match these two representations.

\subsection{Representations by automorphisms of a free group}
Further we consider several representations of $T_n$ by automorphisms of free groups. The first one-parametric family of representations consists of the representations $\phi_k :  T_n \to \operatorname{Aut} (F_n)$ for $k\in \mathbb{Z}$, where $\phi_k$ acts on the generators $\tau_1,\tau_2,\dots,\tau_{n-1}$ of $T_n$ by the following rule 
$$
\phi_k(\tau_i) :  
\begin{cases}
x_i  \to x_{i+1}^{-k} x_i x_{i+1}^{k}, & \\
x_j  \to  x_j,  &  \mbox{if} \quad  j \neq i.
\end{cases}
$$

The second family of representations is the two-parametric family of representations $\psi_{k,l}:  T_n \to \operatorname{Aut} (F_n)$ for $ k,l  \in \mathbb{Z}\setminus \{0\}$, where $\phi_{k,l}$ acts on the generators $\tau_1,\tau_2,\dots,\tau_{n-1}$ of $T_n$ by the following rule 
$$
\psi_{k,l} (\tau_i) :  
\begin{cases}
x_i \to x_i^{-l}x_{i+1}^{-k} x_i x_{i+1}^k x_i^{l}, & \\
x_{i+1} \to x_i^{-l}x_{i+1}x_i^{l},&  \\
x_j \to x_j & \mbox{if} \quad  j \neq i, i+1. 
\end{cases}
$$
It is evidently that $\psi_{k,0} = \phi_k$.

Finally, the representation $\xi :  T_n \to \operatorname{Aut} (F_{3n - 2})$, where the free group $F_{3n-2}$ has the set of free generators 
$$\{ x_{1}, \ldots, x_{n}, y_{1,1}, \ldots, y_{1,n-1}, y_{2,1}, \ldots, y_{2,n-1}\},$$ 
acts on the generators $\tau_1,\tau_2,\dots,\tau_{n-1}$ of $T_n$ by the following rule 
$$
\xi(\tau_i) :
\begin{cases}
x_i \to x_i y_{i,1}, & \\
x_{i+1} \to x_{i+1} y_{i,2}, & \\
x_j \to x_j, & \mbox{if} \quad  j \neq i\\
y_{i,m} \to y_{i,m}, & \mbox{if} \quad i=1, 2 \quad \mbox{and} \quad m=1,\ldots, n-1.
\end{cases}
$$

From the definition it is clear that $\phi_k(\tau_i) = \varepsilon_{i,i+1}^k$ and $\psi_{k, l}(\tau_i) = \varepsilon_{i,i+1}^k\varepsilon_{i+1,i}^l$, therefore, representations $\phi_k$ and $\psi_{k,l}$ act from $T_n$ to the group  of basis conjugating automorphisms $Cb_n$. 
\begin{lemma}\label{easypoint}
Let $w$ be an element in $T_n$, and $s\in \{1,2,\dots,n-1\}$ be a minimal number such that $w$ cannot be written in terms of generators $\tau_1,\tau_2,\dots,\tau_{n-1}$ without $\tau_s$. If $k\neq 0$, then $\phi_k(w)(x_s)\neq x_s$.
\end{lemma}
\begin{proof}Every word $w$ in the group $T_n$ can be written in the normal form
	$$w=\tau_{s_1}^{\alpha_1}\tau_{s_2}^{\alpha_2}\dots\tau_{s_r}^{\alpha_r},$$
where $\alpha_1,\alpha_2,\dots,\alpha_r$ are non-zero integers, and for all $i=1,2,\dots,r-1$ we have either $s_i<s_{i+1}$ or $s_i=s_{i+1}+1$. If the word $w$ is written in the normal form using generators $\tau_{s_1},\tau_{s_2},\dots,\tau_{s_r}$, then $w$ cannot be written (even not in normal form) withour using these generators. From this fact it follows that the word $w$ from the formulation of the lemma can be written in terms of $\tau_{s},\tau_{s+1},\dots,\tau_{n-1}$, and the generator $\tau_s$ is used in the expression of $w$.	
	
	We will prove the lemma using specific induction on $n$. First, we will prove that the lemma is correct in the case when $s=n-1$, then we will assume that the lemma is correct for all numbers $s+1,s+2,\dots,n-1$, and using this conjecture we will prove that the lemma is correct for $s$.

The basis of the induction is obvious: if $s=n-1$, then $w=\tau_{n-1}^{\alpha}$ for $\alpha\neq 0$, and from the definition of $\phi_k$ it is clear that $$\phi_k(w)(x_{n-1})=\phi_k(\tau_{n-1}^{\alpha})(x_{n-1})\neq x_{n-1}.$$ 
Let us assume that the lemma is correct for all numbers $s+1,s+2,\dots,n-1$, and let us prove the lemma for $s$.
	
By contrary, suppose that $w$ is a counter example with the minimal value $|w|$, where $|w|$ denotes the length of $w$, when it is written in terms of generators $\tau_1,\tau_2,\dots,\tau_{n-1}$. This assumption means that $w$ cannot be written in terms of generators $\tau_1,\tau_2,\dots,\tau_{n-1}$ without $\tau_s$ and $\phi_k(w)(x_s)=x_s$. Since $s\in \{1,2,\dots,n-1\}$ is a minimal number such that $w$ cannot be written in terms of generators $\tau_1,\tau_2,\dots,\tau_{n-1}$ without $\tau_s$, the element $w$ can be written in the following form
\begin{equation}\label{welem}w=w_1\tau_s^{\alpha_1}w_2\tau_s^{\alpha_2}\dots w_m\tau_s^{\alpha_m}w_{m+1},
\end{equation}
where $\alpha_1,\alpha_2,\dots,\alpha_m$ are non-zero integers, and $w_1,w_2,\dots,w_{m+1}$ are elements from $T_n$ which can be written in terms of $\tau_{s+1},\tau_{s+2},\dots,\tau_{n-1}$. 

From the definition of the representation $\phi_k$ it is clear that since $w_{m+1}$ can be written in terms of generators $\tau_{s+1},\tau_{s+2},\dots,\tau_{n-1}$, and $\phi_k(w)(x_s)=x_s$ then for the element $$\widetilde{w}=w_{m+1}ww_{m+1}^{-1}=w_{m+1}w_1\tau_s^{\alpha_1}w_2\tau_s^{\alpha_2}\dots w_m\tau_s^{\alpha_m}$$
we have the equality $\phi_k(\widetilde{w})(x_s)=x_s$. Since $|\widetilde{w}|\leq |w|$, considering the element $\widetilde{w}$ instead of $w$ we can assume that the equality $w_{m+1}=1$ holds in (\ref{welem}), i.~e.
\begin{equation}\label{welem2}w=w_1\tau_s^{\alpha_1}w_2\tau_s^{\alpha_2}\dots w_m\tau_s^{\alpha_m}.
\end{equation}
Since the element $w_1$ can be written in terms of generators $\tau_{s+1},\tau_{s+2},\dots,\tau_{n-1}$ from the definition of the representation $\phi_k$ it is clear that 
$$\phi_k(w_1\tau_s^{\alpha_1}w_2\tau_s^{\alpha_2}\dots w_m\tau_s^{\alpha_m})(x_s)=\phi_k(\tau_s^{\alpha_1}w_2\tau_s^{\alpha_2}\dots w_m\tau_s^{\alpha_m})(x_s).$$
Since $|\tau_s^{\alpha_1}w_2\tau_s^{\alpha_2}\dots w_m\tau_s^{\alpha_m}|\leq |w_1\tau_s^{\alpha_1}w_2\tau_s^{\alpha_2}\dots w_m\tau_s^{\alpha_m}|=|w|$, considering the element $\tau_s^{\alpha_1}w_2\tau_s^{\alpha_2}\dots w_m\tau_s^{\alpha_m}$ instead of $w$ we can assume that the equality $w_{1}=1$ holds in (\ref{welem}), i.~e.
\begin{equation}\label{welem3}w=\tau_s^{\alpha_1}w_2\tau_s^{\alpha_2}w_3\dots w_m\tau_s^{\alpha_m}.
\end{equation}
Consider the following elements in $T_n$:
\begin{align*}
u_1&=w_2w_3w_4\dots w_m,\\
u_2&=w_3w_4\dots w_m,\\
u_3&=w_4\dots w_m,\\
&~~\vdots\\
u_{m-1}&=w_m,\\
u_m&=1.
\end{align*}
From the definition of the representation $\phi_k$ and direct calculations it follows that
\small
\begin{equation}\label{actiononxs}
\phi_k(w)(x_s)=x_s^{\left(\phi_k(u_m)(x_{s+1}^{\beta_m})\right)   \left(\phi_k(u_{m-1})(x_{s+1}^{\beta_{m-1}})\right)\dots \left(\phi_k(u_2)(x_{s+1}^{\beta_2})\right)\left(\phi_k(u_1)(x_{s+1}^{\beta_1})\right)},
\end{equation}
\normalsize
where $\beta_i=k\alpha_i$ for $i=1,2,\dots,m$. Since $\phi_k(w)(x_s)=x_s$ and (from the definition of $\phi_k$) the element  
$$\left(\phi_k(u_m)(x_{s+1}^{\beta_m})\right)   \left(\phi_k(u_{m-1})(x_{s+1}^{\beta_{m-1}})\right)\dots \left(\phi_k(u_2)(x_{s+1}^{\beta_2})\right)\left(\phi_k(u_1)(x_{s+1}^{\beta_1})\right)$$ 
can be written in terms of $x_{s+1},x_{s+2},\dots,x_n$, from equality (\ref{actiononxs}) we conclude the equality 
\small
\begin{equation}
	\label{topeq1}\left(\phi_k(u_m)(x_{s+1}^{\beta_m})\right)   \left(\phi_k(u_{m-1})(x_{s+1}^{\beta_{m-1}})\right)\dots \left(\phi_k(u_2)(x_{s+1}^{\beta_2})\right)\left(\phi_k(u_1)(x_{s+1}^{\beta_1})\right)=1.
\end{equation}
\normalsize
 From the definition of $\phi_k$ it is easy to see that for $i=1,2,\dots,m$ we have the equality \begin{equation}\label{whereviis}\left(\phi_k(u_i)(x_{s+1}^{\beta_i})\right)=\left(x_{s+1}^{\beta_i}\right)^{v_i},
 \end{equation}
 where $v_1,v_2,\dots,v_m$ are elements in $F_n$ which can be written in terms of $x_{s+2},x_{s+3},\dots,x_n$. From the equality $u_m=1$ we have $$\left(\phi_k(u_m)(x_{s+1}^{\beta_m})\right)=x_{s+1}^{\beta_m}$$ and taking into account equality (\ref{whereviis}) we can rewrite equality (\ref{topeq1}) in the following form
 \begin{equation}
 	\label{topeq2}
 	x_{s+1}^{\beta_m}\left(x_{s+1}^{\beta_{m-1}}\right)^{v_{m-1}}\dots \left(x_{s+1}^{\beta_2}\right)^{v_2}\left(x_{s+1}^{\beta_1}\right)^{v_1}=1.
 \end{equation}
Due to the fact that for $i=1,2,\dots,m$ the elements $v_2,v_3,\dots,v_m\in F_n$ can be written in terms of $x_{s+2},x_{s+3},\dots,x_n$, in order for the element $x_{s+1}^{\beta_m}$ in equality (\ref{topeq2}) to be reduced with something, there must be a number $t$ such that $\left(x_{s+1}^{\beta_t}\right)^{v_t}=x_{s+1}^{\beta_t}$ and 
 \begin{equation}
	\label{topeq3}
	x_{s+1}^{\beta_m}\left(x_{s+1}^{\beta_{m-1}}\right)^{v_{m-1}}\dots \left(x_{s+1}^{\beta_{t+1}}\right)^{v_{t+1}}\left(x_{s+1}^{\beta_t}\right)^{v_t}=1.
\end{equation}
From equality (\ref{topeq3}) it follows that 
\begin{align}
	\notag x_s&=x_s^{x_{s+1}^{\beta_m}\left(x_{s+1}^{\beta_{m-1}}\right)^{v_{m-1}}\dots \left(x_{s+1}^{\beta_{t+1}}\right)^{v_{t+1}}\left(x_{s+1}^{\beta_t}\right)^{v_t}}\\
	\notag &=x_{s}^{\left(\phi_k(u_m)(x_{s+1}^{\beta_m})\right)   \left(\phi_k(u_{m-1})(x_{s+1}^{\beta_{m-1}})\right)\dots \left(\phi_k(u_t)(x_{s+1}^{\beta_t})\right)}\\
\label{shorterelem}	&=\phi_k(\tau_s^{\alpha_t}w_{t+1}\tau_s^{\alpha_{t+1}}w_{t+2}\dots\tau_s^{\alpha_{m-1}}w_{m}\tau_s^{\alpha_m})(x_s)
\end{align}
If $t\neq 1$, then $$|\tau_s^{\alpha_t}w_{t+1}\tau_s^{\alpha_{t+1}}w_{t+2}\dots\tau_s^{\alpha_{m-1}}w_{m}\tau_s^{\alpha_m}|<|\tau_s^{\alpha_1}w_2\tau_s^{\alpha_2}w_3\dots w_m\tau_s^{\alpha_m}|=|w|$$
and equality (\ref{shorterelem}) contradicts the fact that $w$ is a  counter example to the lemma with the minimal value $|w|$, hence we have $t=1$, what means that 
$$\phi_k(u_1)(x_{s+1}^{\beta_1})=\left(x_{s+1}^{\beta_1}\right)^{v_1}=x_{s+1}^{\beta_1}.$$
This equality together with the definition of $\phi_k$ implies that the eqiality $\phi_k(u_1)(x_{s+1})=x_{s+1}$ holds. 

Let us calculate $\phi_k(w)(x_{s+1})$. From the definition of $\phi_k$ we have the following equality
\begin{multline*}\phi_k(w)(x_{s+1})=\phi_k(\tau_s^{\alpha_1}w_2\tau_s^{\alpha_2}w_3\dots w_m\tau_s^{\alpha_m})(x_{s+1})=\\
	=\phi_k(w_2w_3\dots w_m)(x_{s+1})=\phi_k(u_1)(x_{s+1})=x_{s+1}.
	\end{multline*}
So, now we have the following equalities
\begin{align}
\label{almost}w&=\tau_s^{\alpha_1}w_2\tau_s^{\alpha_2}w_3\dots w_m\tau_s^{\alpha_m},\\
\label{almost1}\phi_k(w)(x_s)&=x_s,\\
\label{almost2}\phi_k(w)(x_{s+1})&=x_{s+1}.
\end{align}
Note that in equality  (\ref{almost}) the number $m$ can not be equal to $1$, since if $m=1$, then $w=\tau_s^{\alpha_1}$, and it is clear that $\phi_k(\tau_s^{\alpha_1})(x_s)\neq x_s$ for $\alpha_1\neq 0$. Therefore $m>1$.

Consider the element $$\widetilde{w}=\tau_s^{-\alpha_1}w\tau_s^{\alpha_1}=w_2\tau_s^{\alpha_2}w_3\dots w_m\tau_s^{\alpha_m+\alpha_1}.$$ 
From equalities (\ref{almost1}), (\ref{almost2}) we have the equality
$$\phi_k(\widetilde{w})(x_s)=x_s.$$
Since the element $w_2$ can be written in terms of generators $\tau_{s+1},\tau_{s+2},\dots,\tau_{n-1}$ from the definition of the representation $\phi_k$ it is clear that 
$$x_s=\phi_k(w_2\tau_s^{\alpha_2}w_3\dots w_m\tau_s^{\alpha_m+\alpha_1})(x_s)=\phi_k(\tau_s^{\alpha_2}w_3\dots w_m\tau_s^{\alpha_m+\alpha_1})(x_s)$$
and since $|\tau_s^{\alpha_2}w_3\dots w_m\tau_s^{\alpha_m+\alpha_1}|<|w_2\tau_s^{\alpha_2}w_3\dots w_m\tau_s^{\alpha_m+\alpha_1}|$ we must have that the element $\tau_s^{\alpha_2}w_3\dots w_m\tau_s^{\alpha_m+\alpha_1}$ doesn't depent on $\tau_s$. It means that $m=2$ and $\alpha_1+\alpha_m=0$, and hence 
$$w=\tau_s^{\alpha_1}w_2\tau_s^{-\alpha_1},$$
where the word $w_2$ can be written in terms of generators $\tau_{s+1},\tau_{s+2},\dots,\tau_n$ and $\phi_k(w_2)(x_{s+1})=x_{s+1}$. By the unduction conjecture the word $w$ can be written in terms of $\tau_{s+2},\tau_{s+3},\dots,\tau_{n-1}$, therefore from the definint relations of $T_n$ we have $w=\tau_s^{\alpha_1}w_2\tau_s^{-\alpha_1}=w_2$, i.~e. $w$ can be written without using $\tau_{s}$, what contradicts conditions of the lemma.
\end{proof}

The following theorem is the main result of this section.
\begin{theorem}\label{psikfa}The following statements hold.
\begin{enumerate}	
	\item If $k\neq 0$, then the representation $\phi_{k} : T_n \to Cb_n$ is faithful.
	
	\item If $k\neq 0$, then the representation $\psi_{k, k} : T_n \to Cb_n$ is faithful.
	
    \item The representation $\xi :  T_n \to \operatorname{Aut}(F_{3n - 2})$ is faithful for $n = 3$ and has non-trivial kernel for $n \geqslant 4$. 
    \end{enumerate}
\end{theorem}
\begin{proof}(1) Let $w$ be an element from ${\rm Ker}(\phi_k)$. Since for all $s\in\{1,2,\dots,n\}$ we have $\phi_k(w)(x_s)=x_s$, from Lemma~\ref{easypoint} it follows that in the representation of $w$ in terms of $\tau_1,\tau_2,\dots,\tau_{n-1}$ there are no generators $\tau_1,\tau_2,\dots,\tau_{n-1}$, hence, $w=1$.
	
(2) From result of Collins \cite{C} follows that the  subgroup $\langle \sigma_1^2, \sigma_2^2, \dots , \sigma_{n-1}^2 \rangle$ of $B_n$ has the presentation
$$
\langle \sigma_1^2, \sigma_2^2, \dots , \sigma_{n-1}^2 ~|~ \sigma_i^2\sigma_j^2 = \sigma_j^2\sigma_i^2 \text{ for } |i - j| \geqslant 2  \rangle,
$$
hence the map $\tau_i\to\sigma_i^2$ for $i=1,2,\dots,n-1$ gives an isomorphism $$T_n\to \langle \sigma_1^2, \sigma_2^2, \dots , \sigma_{n-1}^2 \rangle<P_n.$$ From the main result of the paper \cite{Shpil} it follows that if $k\neq 0$, then the map $\sigma_i^2\mapsto \varepsilon_{i,i+1}^k\varepsilon_{i+1,i}^k=\psi_{k,k}(\tau_i)$ for $i=1,2,\dots,n-1$ gives a faithful representation of $P_n$ to ${\rm Aut}(F_n)$. Since $\psi_{k,k}$ is the superposition of these two maps 
$$T_n\to \langle \sigma_1^2, \sigma_2^2, \dots , \sigma_{n-1}^2 \rangle\to {\rm Aut}(F_n),$$
the representation $\psi_{k,k}$ is faithful.

	(3) For $n = 3$ the group $\xi(T_3)$ acts on the generator $x_2$ as the free group $F_2$ which is isomorphic to $T_3$, therefore  the representation $\xi :  T_n \to \operatorname{Aut}(F_{3n - 2})$ is faithful for $n=3$.  If $n \geqslant 4$, then from direct calculations it is easy to see that the element
	$$
	\tau_1\tau_2\tau_3\tau_2^{-1}\tau_1^{-1}\tau_2\tau_3^{-1}\tau_2^{-1}
	$$
	belongs to the kernel of $\xi$.
\end{proof}

Let us now consider extensions of the representations $\phi_k,\psi_{k,l}$ from the group $T_n$ to the group $UB_n$, 
\begin{align*}
\widetilde{\phi}_k &: UB_n \to \operatorname{Aut} (F_n),&&k\in\mathbb{Z},\\
\widetilde{\psi}_{k,l} &: UB_n \to \operatorname{Aut}(F_n), &&k,l  \in \mathbb{Z}
\end{align*}
which act on the generators $\sigma_1,\sigma_2,\dots,\sigma_{n-1}$ of $UB_n$ as Artin representation presented in Section~\ref{braidsandautomorphisms}. The following statement says that despite the fact that according to Theorem~\ref{psikfa} the representation $\phi_k:T_3\to {\rm Aut}(F_3)$ is faithful for $k\neq 0$, the extension $\widetilde{\phi}_k : UB_3 \to \operatorname{Aut}(F_3)$ is not faithful.

\begin{proposition} 
	The representation $\widetilde{\phi}_k : UB_3 \to \operatorname{Aut}(F_3)$ has non-trivial kernel for any $k \in \mathbb{Z}$.
\end{proposition}
\begin{proof}
	Using direct calculations it is easy to see that the element 
	$$\sigma_1 \tau_1^k \sigma_2 \tau_2^k \sigma_1 \tau_1^k \tau_2^{-k} \sigma_2^{-1} \tau_1^{-k} \sigma_1^{-1} \tau_2^{-k} \sigma_2^{-1}$$
	belongs to the kernel of $\widetilde{\phi}_k$.
\end{proof}

\begin{problem}
Is it true that  $\widetilde{\psi}_{k,k}$ is faithful for $k \not= 0$?
\end{problem}

At the end of this subsection we introduce a representation $$\Psi_{k,l} : UB_n \to {\rm Aut}(F_{n+1}),$$ where $F_{n+1} = \langle x_1, x_2, \dots , x_{n}, y \rangle$ which generalizes the representation $\widetilde{\psi}_{k,l}$, and which acts on the generators $\sigma_1,\sigma_2,\dots,\sigma_{n-1},\tau_1,\tau_2,\dots,\tau_{n-1}$ of $UB_n$ by the rules
\begin{align*}
\Psi_{k,l}(\sigma_{i}) &: \begin{cases}
x_{i} \mapsto x_{i} x_{i+1}x_i^{-1},&\\
x_{i+1} \mapsto x_{i},&\\
x_{j} \mapsto x_{j}, &  j \neq i,i+1,\\
y \mapsto y, &
\end{cases}\\
\Psi_{k,l} (\tau_i) &:  
\begin{cases}
x_i \to y^{-1} x_i^{-l}x_{i+1}^{-k} x_i x_{i+1}^k x_i^{l} y, & \\
x_{i+1} \to y^{-1} x_i^{-l}x_{i+1}x_i^{l} y,&  \\
x_j \to x_j & \mbox{if} \quad  j \neq i, i+1, \\ 
y \mapsto y, &
\end{cases}
\end{align*}
for $i=1, 2, \ldots, n-1$. It is easy to see that $\Psi_{k,l}$ generalizes the representation $\widetilde{\psi}_{k,l}$: if we put $y=1$, then $\Psi_{k,l}$ goes to $\widetilde{\psi}_{k,l}$.

\begin{problem}
Do there exist some integers $k,l$ such that the representation $\Psi_{k,l} : UB_n \to {\rm Aut}(F_{n-1})$ is faithful for all $n > 2$?
\end{problem}

\subsection{Linear representations}
As a corollary from the previous section, in the present section we construct a linear representations of the groups $T_n$. In order to do it we use the Magnus approach for construction of linear representations from representations by automorphisms of free groups. This approach is described, for example, in \cite[Section~4.1]{alotofauthors}. In the present section we will just introduce the final representation we got.

 Let $M$ be the free module with the basis $\{e_1, e_2, \ldots, e_{n-1}\}$ over the ring $R = \mathbb{Z}[t_1^{\pm 1}, t_2^{\pm 1}, \ldots, t_{n}^{\pm 1}]$ of Laurent polynomials with variables $t_1,t_2,\dots,t_n$. Define the map $\Theta : T_n \to \mathrm{Aut}(M)$ which acts on the generators $\tau_1,\tau_2,\dots,\tau_{n-1}$ by the following rule
$$
\Theta(\tau_i) : \begin{cases}
e_{i} \longmapsto t_{i+1}^{-1} e_i + t_{i+1}^{-1} (t_i-1) e_{i+1}, &  \\
e_{j} \longmapsto e_j, &j\not=i.
\end{cases}
$$
The following statement obviously follows from the fact that the matrices of the elements $\Theta(\tau_i)$ in the basis $e_1,e_2,\dots,e_n$ are triangular.
\begin{proposition}
The image $\operatorname{Im}  \Theta$ is a  solvable group. Hence, for $n>2$ the representation $\Theta$ has a non-trivial kernel.
\end{proposition}

\subsection{Universal pure braid group} In this section we study universal pure braid groups on small amount of strands. Consider the homomorphism $\pi : UB_n \to S_n$ given on generators of $UB_n$ by the rule $\pi(\sigma_i) = \pi(\tau_i) = (i, i+1)$ for $i = 1, 2, \ldots, n-1$. 
The kernel $\operatorname{Ker}(\pi)$ of this homomorphism is known as the \emph{universal pure braid group}, denoted by $U{P_n}$. The following short exact sequence holds: 
$$
1 \to UP_{n} \to UB_{n} \stackrel{\pi}{\to} S_{n} \to  1.
$$ 
It is obvious that $P_n$ is a subgroup of $U{P_n}$. Hence, the last group contains all generators and relations of $P_n$. 

\begin{theorem} \label{theorem5.1}
The following properties hold:  

1) $UB_2 = B_2 * T_2 \cong F_2$, ~~$UB_3 = B_3 * T_3 = B_3 * F_2$,  

2) $U{P_2} =  F_3$,~~$UP_3 = {P_3} * {F_{12}}$.
\end{theorem}

\begin{proof} 
1) Follows from the presentations of $UB_2$ and $UB_3$.

2)
We will find the presentation of $UP_{3} = \operatorname{Ker} (\pi)$ by the Reidemeister-Shreier method using the Shreier set $\{1, \sigma_{1}, \sigma_{2}, \sigma_{1} \sigma_{2}, \sigma_{2} \sigma_{1}, \sigma_{1} \sigma_{2} \sigma_{1}, \sigma_{2} \sigma_{1} \sigma_{2} \}$. Since the restriction of homomorphism $\pi$ on $B_{3} \subset UB_{3}$ has $P_{3}$ as a kernel, it is enough to consider only generators $\tau_{1}, \tau_{2}$ of $UB_{3}$. Hence we get  
\begin{align*}
& S_{1,\tau_1} = \tau_1 \cdot (\overline{\tau_1})^{-1} = \tau_1 \cdot \sigma_1^{-1},\\
& S_{1,\tau_2} = \tau_2 \cdot (\overline{\tau_2})^{-1} = \tau_2 \cdot \sigma_2^{-1},\\
&S_{\sigma_1,\tau_1} = \sigma_1 \tau_1 \cdot (\overline{\sigma_1 \tau_1})^{- 1} = \sigma_1 \tau_1,\\
& S_{\sigma_1,\tau_2} = \sigma_1 \tau_2 \cdot (\overline{\sigma_1 \tau_2})^{-1} = \sigma_1 \tau_2 \sigma_2^{-1} \sigma_1^{-1},\\
& S_{\sigma_2,\tau_1} = \sigma_2 \tau_1 \cdot (\overline{\sigma_2 \sigma_1})^{-1} = \sigma_2 \tau_1 \sigma_1^{-1} \sigma_2^{-1},\\
& S_{\sigma_2,\tau_2} = \sigma_2 \tau_2,\\
& S_{\sigma_1 \sigma_2,\tau_1} = \sigma_1 \sigma_2 \tau_1 \sigma_1^{-1} \sigma_2^{-1} \sigma_1^{-1},\\
&S_{\sigma_1 \sigma_2,\tau_2} = \sigma_1 \sigma_2 \tau_2 \sigma_1^{-1},\\
& S_{\sigma_2 \sigma_1,\tau_1} = \sigma_2 \sigma_1 \tau_1 \sigma_2^{-1},\\
& S_{\sigma_2 \sigma_1,\tau_2} = \sigma_2 \sigma_1 \tau_2 \sigma_2^{-1} \sigma_1^{-1} \sigma_2^{-1},\\
& S_{\sigma_1 \sigma_2 \sigma_1, \tau_1} = \sigma_1 \sigma_2 \sigma_1 \tau_1 \sigma_2^{-1} \sigma_1^{-1},\\
& S_{\sigma_1 \sigma_2 \sigma_1, \tau_2} = \sigma_1 \sigma_2 \sigma_1 \tau_2 \sigma_2^{-1}\sigma_1^{-1} \sigma_2^{-1} \sigma_1^{-1}.\\
\end{align*}
Since, $T_3$ is free, all of its subgroups are also free. Hence, we have the decomposition for $UP_3$. The group $UP_2$ is its subgroup and it is generated by 
$$
a_{12} = \sigma_1^2,~~S_{1,\tau_1} = \tau_1  \sigma_1^{-1},~~~S_{\sigma_1,\tau_1}  = \sigma_1 \tau_1.
$$
We do not have relations between them. Hence, $UP_2$ is free of rank $3$.\end{proof}

\begin{corollary} The following properties hold:
\begin{itemize}
\item[1)] $UB_3$ is linear; 
\item[2)] $UP_3$ is residually nilpotent.
\end{itemize}
\end{corollary}

\begin{proof}
1) Follows from the fact that the free product of two linear groups is linear.

2) Follows from the fact that the free product of two residually nilpotent groups that do not have torsion is residually nilpotent.
\end{proof}

\begin{remark}
If $n>3$, then we can not say that $UB_n$ is the free product of $B_n$ and $T_n$. In this case we have mixed relations of the form
$$
[\sigma_{i}, \tau_j] = 1,~~|i-j|>1.
$$
Also, if $n>3$  we can not say that $UP_n$ is the free product of $P_n$ and a free group.
In this case we have relations of the form
$$
[\tau_{i}^2, \tau_j^2] = 1,~~|i-j|>1.
$$
\end{remark}

\begin{question}
For $n>3$ find a set of generators and defining relations of the group $UP_n$.
\end{question}

\section*{Acknowledgements} 
The results were obtained during the program on knots and braid groups advised by V.~Bardakov, T.~Nasybullov and A.~Vesnin in the frame of The First Workshop at the Mathematical Center in Akademgorodok supported by the Ministry of Science and Higher Education of the Russian Federation (agreement no. 075--2019--1675).

\bigskip

Valeriy Bardakov$^{*,\dag, \ddag}$ (bardakov@math.nsc.ru),   

Ivan  Emel'yanenkov$^*$ (i.emelianenkov@g.nsu.ru),  

Maxim Ivanov$^*$ (m.ivanov2@g.nsu.ru),  

Tatyana Kozlovskaya$^{\dag}$ (t.kozlovskaya@math.tsu.ru),  

Timur Nasybullov$^{*,\dag, \ddag}$ (timur.nasybullov@mail.ru),  

Andrei Vesnin$^{*,\dag, \ddag}$ (vesnin@math.nsc.ru)

~\\
$^*$ Novosibirsk State University, Pirogova 1, 630090 Novosibirsk, Russia,\\
$^{\dag}$ Tomsk State University, Lenin Ave. 36, 634050 Tomsk, Russia,\\
$^{\ddag}$ Sobolev Institute of Mathematics, Acad. Koptyug Ave. 4, 630090 Novosibirsk, Russia,\\

\end{document}